\documentclass[11pt]{amsart}
\usepackage{amsmath,amsthm, amscd, amssymb, amsfonts, mathtools}
\usepackage[all]{xy}
\usepackage{mathrsfs}
\usepackage{enumitem}
\usepackage[T2A,T1]{fontenc}
\usepackage{rotating}
\usepackage{multicol}

\usepackage{hyperref}

\allowdisplaybreaks

\usepackage[ansinew]{inputenc}
\usepackage{graphicx,fancyhdr}

\usepackage[dvips, dvipsnames, usenames]{color}

\newcommand{\nucleo}{\mathbf R}
\newcommand{\gen}{\mathbf{g}}
\newcommand{\Gb}{\mathbf G}
\newcommand{\Hb}{\mathbf H}
\newcommand{\Bb}{\mathbf B}

\newcommand{\verma}{M}

\numberwithin{equation}{section}
\newtheorem{theorem}{Theorem}[section]
\newtheorem{lemma}[theorem]{Lemma}
\newtheorem{coro}[theorem]{Corollary}

\newtheorem{prop}[theorem]{Proposition}

\theoremstyle{definition}
\newtheorem{definition}[theorem]{Definition}

\theoremstyle{remark}
\newtheorem{remark}[theorem]{Remark}

\newcommand{\pf}{\begin{proof}}
\newcommand{\epf}{\end{proof}}

\newcommand{\spl}{\mathfrak{sl}}

\newcommand{\ku}{ \Bbbk}
\newcommand{\fp}{\mathbb F_p}

\newcommand{\kut}{ \ku^{\times}}

\newcommand{\x}{\mathtt{x}}
\newcommand{\yt}{\mathtt{y}}

\newcommand{\wtoba}{\widetilde{\toba}}
\newcommand{\ttoba}{\widetilde{\mathfrak B}}

\newcommand{\I}{\mathbb I}

\newcommand{\N}{\mathbb N}

\newcommand{\Z}{\mathbb Z}

\newcommand{\sx}{\mathsf{x}}
\newcommand{\sy}{\mathsf{y}}
\newcommand{\cA}{\mathcal{A}}

\newcommand{\cO}{\mathcal{O}}
\newcommand{\cE}{\mathcal{E}}

\newcommand{\D}{\mathcal{D}}

\newcommand{\cI}{\mathcal{I}}

\newcommand{\Pc}{{\mathcal P}}
\newcommand{\cR}{\mathcal{R}}
\newcommand{\Ss}{{\mathcal S}}

\newcommand{\cV}{\mathcal{V}}

\newcommand{\Alg}{\Hom_{\text{alg}}}
\newcommand{\Aut}{\operatorname{Aut}}

\newcommand{\car}{\operatorname{char}}

\newcommand{\id}{\operatorname{id}}
\newcommand{\gr}{\operatorname{gr}}
\newcommand{\GK}{\operatorname{GKdim}}
\newcommand{\Hom}{\operatorname{Hom}}

\newcommand{\Irr}{\operatorname{Irrep}}
\newcommand{\Ind}{\operatorname{Ind}}

\newcommand{\ydk}{{}^{K}_{K}\mathcal{YD}}

\newcommand{\toba}{\mathscr{B}}
\newcommand{\ot}{\otimes}

\newcommand{\ydG}{{}^{\ku \Gamma }_{\ku \Gamma }\mathcal{YD}}

\newcommand{\ydGd}{{}^{\ku^\Gamma }_{\ku^\Gamma }\mathcal{YD}}
\newcommand{\ydtildeGd}{{}^{\ku[\zeta]}_{\ku[\zeta]}\mathcal{YD}}
\newcommand{\ydtildeG}{{}^{\ku \widetilde{\Gamma} }_{\ku \widetilde{\Gamma} }\mathcal{YD}}
\DeclareRobustCommand{\stirling}{\genfrac []{0pt}{}}

\newcommand{\rightarrowdbl}{\rightarrow\mathrel{\mkern-14mu}\rightarrow}

\newcommand{\xrightarrowdbl}[2][]{%
\xrightarrow[#1]{#2}\mathrel{\mkern-14mu}\rightarrow
}

\newcounter{tabla}\stepcounter{tabla}

\begin{document}

\title[On the restricted Jordan plane]
{On the restricted Jordan plane in odd characteristic}

\author[Nicol\'as Andruskiewitsch and H\'ector Pe\~na Pollastri]
{Nicol\'as Andruskiewitsch and H\'ector Pe\~na Pollastri}

\thanks{The work of N. A. and H. P. P. was partially supported by CONICET and Secyt (UNC)}

\address{ Facultad de Matem\'atica, Astronom\'ia y F\'isica,
Universidad Nacional de C\'ordoba. CIEM -- CONICET. 
Medina Allende s/n (5000) Ciudad Universitaria, C\'ordoba, Argentina}
\email{andrus|hpollastri@famaf.unc.edu.ar}

\begin{abstract}
In positive characteristic the Jordan plane covers a finite-dimensional Nichols algebra that  was described by Cibils, Lauve and Witherspoon 
and we call the restricted Jordan plane.  In this paper  the  characteristic is odd. 
The defining relations of the Drinfeld double of the restricted Jordan  plane are presented and its simple modules are determined.
A Hopf algebra that deserves the name of double of the Jordan plane is introduced and various quantum Frobenius maps are described. 
The finite-dimensional pre-Nichols algebras intermediate between the Jordan plane and its restricted version are classified.
The defining relations of the graded dual of the Jordan plane are given.
\end{abstract} 
\maketitle

\setcounter{tocdepth}{1}
\tableofcontents

\section*{Introduction}
The Jordan plane is a well-known example of a quadratic algebra. It is also a cornerstone in the study of Nichols algebras over abelian groups with finite
Gelfand-Kirillov dimension \cite{aah-triang,AAH-jordan}. 

Let $\ku$ be an algebraically closed field of characteristic $p > 2$ and let $\fp$ be the field of $p$ elements.
Let $\Gamma \simeq \Z/p\Z$ be a cyclic group with a generator $g$, 
written multiplicatively. 
Let $V$ be a Yetter-Drinfeld module over $\ku\Gamma$ with a basis $\{x,y\}$, action $\rightharpoonup$ and coaction $\delta$ given by
\begin{align}\label{Eq:V_as_YD_module}
g\rightharpoonup x &= x, & g\rightharpoonup y &= y+x,&
\delta(x) &= g\ot x,& \delta(y) &= g\ot y.
\end{align}
Thus $V$ is a braided vector space with braiding 
\begin{align*}
c_V(x\ot x) &= x\ot x, & c_V(x\ot y) &= (y+x)\ot x, \\
c_V(y\ot x) &= x \ot y, & c_V(y\ot y) &= (y+x)\ot y.
\end{align*}
By \cite[Theorem 3.5]{clw}, the Nichols algebra $\toba(V)$--that we shall call the \emph{restricted} Jordan plane--is the quotient of $T(V)$ by the ideal generated by
\begin{align*}
x^p, && y^p, && yx-xy+\frac{1}{2}x^2.
\end{align*}
The family $(x^i y^j)_{0 \leq i,j \leq p-1}$ is a basis of $\toba(V)$ that has dimension $p^2$.
The liftings of $\toba(V)$ have been computed in \cite{clw} and the simple modules of these in \cite{ZC}.
The restricted Jordan plane is the starting point to the description of new examples of finite-dimensional Hopf algebras \cite{aah-findim}.
We consider in this paper various Hopf algebras related to $\toba(V)$.

\begin{enumerate}[leftmargin=-3.5pt]
\item Let $D(H)$ be the Drinfeld double of the bosonization $H = \toba(V)\#\ku \Gamma$.
We present the defining relations of $D(H)$ and show that it fits into an exact sequence
$\nucleo \hookrightarrow D(H) \twoheadrightarrow \mathfrak u(\spl_2(\ku))$
where $\nucleo$ is a local commutative Hopf algebra and $\mathfrak u(\spl_2(\ku))$
is the restricted enveloping algebra. We conclude that the simple $D(H)$-modules are the same as those of $\mathfrak u(\spl_2(\ku))$
and we present them as quotients of Verma modules.
See Propositions \ref{prop:double-H}, \ref{prop:double-H-ext} and \ref{prop:irrep-D(H)}, and Theorem \ref{th:irrep-restricted-jordan}.

\medbreak
\item The Jordan plane covers $\toba(V)$. We define a Hopf algebra $\widetilde{D}$ that covers the Drinfeld double $D(H)$. 
The definition of $\widetilde{D}$ makes sense in any characteristic $\neq 2$; 
$\widetilde{D}$ can be thought of the Drinfeld double of the Jordan plane and the map $\widetilde{D} \twoheadrightarrow D(H)$
as a quantum Frobenius map. Indeed let us consider the algebraic groups 
\begin{align*}
\Gb &= (\Gb_a \times \Gb_a) \rtimes \Gb_m, &
\Bb = \left((\Gb_a \times \Gb_a) \rtimes \Gb_m\right) \times \Hb_3
\end{align*}
with suitable semidirect products and where $\Hb_3$ is the Heisenberg group of dimension 3. See Remark \ref{rem:algebraic-group-normal-hopfsubalgebra}
and \eqref{eq:alggroup-B}.
Then there is a short exact sequence of Hopf algebra maps $\cO(\Bb) \hookrightarrow\widetilde{D} \twoheadrightarrow D(H)$
that fits into a commutative diagram
\begin{align}\label{eq:diagram-exact sequences}
\begin{aligned}
\xymatrix{ & \cO(\Gb)  \ar@{^{(}->}[r] \ar@{^{(}->}[d]_{\operatorname{Fr}} & \cO(\Bb) \ar@{^{(}->}[d] \ar@{->>}[r] & \cO(\Gb_a^3) \ar@{^{(}->}[d]
\\
(\star) & \cO(\Gb)  \ar@{^{(}->}[r] \ar@{->>}[d] & \widetilde{D} \ar@{->>}[r] \ar@{->>}[d] & U(\spl_2(\ku)) \ar@{->>}[d]
\\
& \nucleo \ar@{^{(}->}[r] & D(H) \ar@{->>}[r] & \mathfrak u(\spl_2(\ku))
}
\end{aligned}
\end{align}
where all columns and rows are exact sequences. Notice that the exact sequence in the middle row $(\star)$ is available also in characteristic 0. 
See Propositions \ref{prop:Dtilde}, \ref{prop:another-exseq} and \ref{prop:diagram-exact sequences}. 
The algebra $\widetilde{D}$ is PI and a noetherian domain,
see Proposition \ref{prop:ringtheoretical}.

\medbreak 
\item  We classify the  finite-dimensional pre-Nichols algebras intermediate between the Jordan plane and the  restricted Jordan plane.
Precisely, any such finite-dimensional pre-Nichols algebra is isomorphic (as braided Hopf algebra) to 
\begin{align*}
\mathcal{G}(k,\ell)&\coloneqq \wtoba/(y^{p^k}, x^{p^\ell})
\end{align*}
for unique $k, \ell \in \N$.
See Theorem \ref{thm;prenichols}.
We do not know if any finite-dimensional pre-Nichols algebra of $\toba(V)$ is like this. For instance, it is clear that
\begin{align*}
\ku \left\langle x, y|x^p, \quad y^p, \quad (yx-xy+\tfrac{1}{2}x^2)^{p^n}
\right\rangle
\end{align*}
is a pre-Nichols algebra of $\toba(V)$ but we do not know whether it has finite dimension or finite $\GK$.
We also provide new examples of Hopf algebras with finite $\GK$  by bosonization with $\ku\Gamma$. See Corollary \ref{coro:pre-Nichols-bosonization}.

\medbreak
\item We give the generators and defining relations of the graded dual $\cE$ of the Jordan plane. 
See Proposition \ref{prop:graded-dual}. Theorem \ref{thm;prenichols} implies that any 
finite-dimensional post-Nichols algebra of $(V, c_V^{-1})$  contained in $\cE$ is isomorphic (as braided Hopf algebra) to 
$\mathfrak{G}(k,\ell) = \mathcal{G}(k,\ell)^*$ for unique $k, \ell \in \N$.
We also give the generators and defining relations of $\mathfrak{G}(k,\ell)$.
\end{enumerate}

\subsection*{Conventions}
If $\ell < n \in\N_0$, then we set $\I_{\ell, n}=\{\ell, \ell +1,\dots,n\}$, $\I_n = \I_{1, n}$. 
Let $K$ be a Hopf algebra. The space of primitive elements of $K$ is denoted by $\Pc(K)$ and the antipode by $\Ss$ or by $\Ss_K$.
The category of Yetter-Drinfeld modules over $K$ is denoted $\ydk$.

Let $A$ be an algebra. 
The set of isomorphism classes of finite-dimensional simple modules over $A$
is denoted $\Irr A$. We usually denote indistinctly a class in $\Irr A$ and one of its representatives. An element
$x\in A$ is normal if $Ax = xA$. 
If $B$ is a subalgebra of $A$,
then $\Ind_{B}^{A}$ denotes the induction functor $M \mapsto A \otimes_B M$.

The algebra of regular functions on an (affine) algebraic group $G$ is denoted $\cO(G)$.
As usual, $\Gb_a$ is the additive algebraic group $(\ku,+)$ and $\Gb_m$ is the multiplicative algebraic group $(\kut,\cdot)$.

We recall that the unsigned Stirling numbers $\stirling{n}{k}$ are defined as the coefficients of the `raising factorial' polynomial: 
\begin{align*}
[X]^{[n]} = \prod_{i=1}^{n} (X+i-1) = \sum_{k=0}^n \stirling{n}{k} X^k \in \Z[X].
\end{align*}

\section{The double of the restricted Jordan plane}

\subsection{The double}
Here we present by generators and relations the Drinfeld double $D(H)$ that clearly has dimension $p^6$.

\subsubsection{The bosonization of the restricted Jordan plane}
By \cite[Corollary 3.14]{clw}, the bosonization $H = \toba(V)\#\ku\Gamma$ is the $p^3$-dimensional pointed Hopf algebra generated by $x$, $y$ and $g$ with relations
\begin{align}\label{eq:relations-H}
\begin{aligned}
g^p = 1, && gx = xg , && gy = yg+xg,\\
x^p = 0, && y^p = 0, && yx = xy - \frac{1}{2} x^2.
\end{aligned} 
\end{align}
The coproduct is determined by
\begin{align}\label{eq:hopf-structure-double-H-1}
\Delta(g) = g\ot g, && \Delta(x)= x\ot 1 + g\ot x, && \Delta(y) = y\ot 1 + g\ot y.
\end{align}

We shall need the following formulas that hold in $H$:
\begin{align}\label{jordan-relations-between-monomials}
&\begin{aligned}
g^{n}\,y^\ell &= \sum_{k=0}^{\ell} \binom{\ell}{k} (-1)^k\frac{[-2n]^{[k]}}{2^k} \,x^k y^{\ell-k}\,g^n,\\
y^\ell\,x^n &= \sum_{k=0}^{\ell} \binom{\ell}{k}(-1)^{k} \frac{[n]^{[k]}}{2^{k}} x^{n+k}\,y^{\ell-k},
\end{aligned}
& n,\ell&\in\N_0,\end{align}
where $[t]^{[k]}$ denotes the raising factorial $[t]^{[k]} := \prod_{i=1}^{k} (t+i-1)$ for $t\in\ku$ and $k\in\N_0$. 
The first formula appears in \cite{clw} and the second is folklore.

\begin{remark}
	The algebra $H$ is local. The unique maximal ideal is $I\coloneqq \langle x,y, g-1 \rangle$. Hence the only simple representation of
	$H$ is the trivial one. See \cite{ZC} for a study of the representations of $H$ and its liftings.
\end{remark}
\subsubsection{The Drinfeld double}
We start by recalling the definition.
\begin{definition}
Let $L$ be a finite-dimensional Hopf algebra. The Drinfeld double of $L$, denoted by $D(L)$, is a Hopf algebra 
whose underlying coalgebra is $ L\ot L^{* \operatorname{op}}$ and with multiplication and antipode defined as follows.

Let $h\bowtie f:= h \otimes f$ in $D(L)$ for all $f\in L^{* \operatorname{op}} = L^*$ and $h\in L$. Then 
\begin{align*}
(h\bowtie f)(h'\bowtie f') &= \big\langle f_{(1)},h'_{(1)}\big\rangle \big\langle f_{(3)}, \Ss(h'_{(3)}) \big\rangle (hh'_{(2)}\bowtie f'f_{(2)}),
\\
\Ss_{D(L)}(h\bowtie f) &= (1\bowtie \Ss^{-1}(f))(\Ss(h)\bowtie \varepsilon),
\end{align*}
where $fr = m(f\ot r)$ is the multiplication in $L^*$ rather than in $L^{* \operatorname{op}}$. 
\end{definition}

Our goal is to present $D(H)$; for this we start with $D(\ku\Gamma) = \ku \Gamma \otimes \ku^{\Gamma}$.
We need to describe $\ku^{\Gamma}$ suitably.
The polynomial ring $\ku[X]$ is a Hopf algebra with $X$ primitive.
Let $(\delta_k)_{ k\in \fp}\subseteq \ku^{\Gamma}$ be the dual basis of $(g^k)_{ k\in \fp}$ and 
\begin{align*}
\zeta =\sum_{k \in \fp} k \delta_k.
\end{align*}
The following result is well-known; of course it is crucial that $\car \ku = p$. 

\begin{lemma} The map $ X \mapsto \zeta$ gives an isomorphism of
Hopf algebras $\ku[X]/(X^p-X) \simeq \ku^{\Gamma}$; and
$(\zeta^i)_{i \in \I_{0,p-1}}$ is a basis of $\ku^\Gamma$. \qed
\end{lemma}

\begin{lemma}\label{double-kgamma} The algebra $D(\ku\Gamma)$ is presented by generators $g,\zeta$ and relations
\begin{align}\label{relations-Dgamma}
g^p &= 1, & \zeta^p &= \zeta, & g\zeta = \zeta g.
\end{align}
The Hopf algebra structure is determined by 
\begin{align}\label{eq:hopf-structure-double-H-2}
\Delta(g)&= g\ot g, & \Delta(\zeta) &=\zeta\ot 1 + 1 \ot \zeta.
\end{align} 
\end{lemma}

\begin{proof}
Let $A$ be the algebra generated by $\{\tilde{g},\tilde{\zeta}\}$ with relations \eqref{relations-Dgamma}. 
Since $\zeta$ and $g \in D(\ku\Gamma)$ satisfy \eqref{relations-Dgamma}, 
we have an epimorphism $A\twoheadrightarrow D(\ku\Gamma)$. But $A$ is linearly generated by
$(\tilde{g}^k\tilde{\zeta}^{\ell})_{\substack{\ell\in\I_{0,p-1}\\k \in \fp}}$, 
so $\dim A \leq p^2 = \dim D(\ku\Gamma)$.
\end{proof}

Next we describe $H^*$. We have morphisms of Hopf algebras $H\overset{\pi }{\underset{\iota}{\rightleftarrows}} \ku \Gamma$ 
such that $\pi \iota = \id$; dualizing we get $H^*\overset{\iota^*}{\underset{\pi ^*}{\rightleftarrows}} \ku^\Gamma$ with $\iota^*\pi ^* = \id$.
Hence $H^* \simeq \cR\# \ku^\Gamma$ where $\cR = \big(H^* \big)^{\operatorname{co} \iota^*} \simeq \toba(W)$, see e.g. \cite[2.3]{Beattie}.
Here $W \in \ydGd$ is $\simeq V^*$. 

\begin{lemma}
$H^{* \operatorname{op}}$ is presented by generators $u$, $ v$ and $\zeta$ with relations
\begin{align}\label{relations-H-dual}
\begin{aligned}
v^p &= 0, &u^p &= 0 , & v u &= u v-\frac{1}{2} u^2,\\
v\zeta &= \zeta v + v, &  u\zeta&= \zeta u + u, & \zeta^p &= \zeta.
\end{aligned}
\end{align}
A basis of $H^{* \operatorname{op}}$ is $(\zeta^k u^i v^j)_{i,j,k\in\I_{0, p-1}}$. The comultiplication is given by
\begin{align}\label{eq:hopf-structure-double-H-3}
\begin{aligned}
\Delta( u) &= u\ot 1 + 1 \ot u, & \Delta(\zeta) &= \zeta\ot 1 + 1 \ot \zeta, \\
\Delta( v) & =  v\ot 1 + 1\ot v +\zeta\ot u. &&
\end{aligned} 
\end{align}
\end{lemma}

\begin{proof}
Let $(e_{i,j,k})_{\substack{i,j\in\I_{0, p-1} \\ k\in \fp }}$ be the basis of $H^{* \operatorname{op}}$ dual to 
$(x^iy^jg^{k})_{\substack{i,j\in\I_{0, p-1} \\ k\in \fp }}$. 
We identify $\ku^\Gamma$ with a subalgebra of $H^*$, so $\zeta = \sum_{k\in \fp} k e_{0,0,k}$. 
We first claim that the following elements are primitive in $\cR$:
\begin{align*}
u = \sum_{k \in \fp} e_{0,1,k}, && v = \sum_{k \in \fp} e_{1,0,k}.
\end{align*}
Indeed, we compute in $H^{*}$:
\begin{align*}
\Delta(e_{1,0,h}) &= \sum_{i \in \fp} e_{1,0,i} \ot e_{0,0,h-i} + e_{0,0,i} \ot e_{1,0,h-i}+ i\,e_{0,0,i} \ot e_{0,1,h-i}, \\
\Delta(e_{0,1,h}) &= \sum_{i \in \fp} e_{0,1,h-i} \ot e_{0,0,i} + e_{0,0,i} \ot e_{0,1,h-i}.
\end{align*}
Hence \eqref{eq:hopf-structure-double-H-3} holds by a straightforward calculation. Thus
$ u, v \in \cR$; clearly they are primitive in $\cR$ and linearly independent.
Hence they form a basis of $W$. Explicitly its structure as Yetter-Drinfeld module is given by
\begin{align}\label{Eq:W_as_YD_module}
\zeta\rightharpoonup u &= u, & \zeta\rightharpoonup v &= v,&
\delta( u) &= 1\ot u,& \delta( v) &= 1\ot v+\zeta\ot u,
\end{align}
while the braiding $c_W$ is determined by
\begin{align*}
c_W( u\ot u) &= u\ot  u, & c_W( u\ot v) &= v\ot u, \\
c_W( v\ot u) &= u \ot ( v+ u), & c_W( v\ot v) &= v\ot ( v+ u).
\end{align*}
Thus we have an isomorphism of braided vector spaces $(V, c_V^{-1}) \to (W, c_W)$ that sends $x\mapsto -u$, $y\mapsto v$. 
The Nichols algebra $\toba(W)$ has the same relations as $\toba(V)$ and is finite-dimensional. 
The relations \eqref{relations-H-dual} are satisfied in $H^{*\operatorname{op}}$, but $\toba(V)$ is not isomorphic to $\toba(W)$
as braided Hopf algebras because $(V,c_V)$ and $(V,c_v^{-1})$ are not isomorphic as braided vector spaces.

Now, let $A$ be the algebra presented by generators $U=\{ u, v,\zeta\}$ and relations \eqref{relations-H-dual}. By the argument above we have $A\twoheadrightarrow H^{*\operatorname{op}}$, therefore $\dim A\geq p^3$. For this to be an isomorphism, 
it is enough to show that $A$ is linearly generated by $(\zeta^i u^j v^k)_{i,j,k\in\I_{0, p-1}}$. 
We define a total order in $U$ by declaring $\zeta< u< v$. 
By the defining relations, any product $ab$ with $a,b\in U$ and $b < a$ can be written as a linear combination of monomials $c_1\cdots c_s$ with $c_1\leq c_2\leq\cdots\leq c_s\in U$.
The claim follows from this together with the relations $u^p = 0$, $v^p = 0$, $\zeta^p = \zeta$.
Hence $A\simeq H^{*\operatorname{op}}$.
\end{proof}

\begin{prop}\label{prop:double-H}
The algebra $D(H)$ is presented by generators $ u, v,\zeta,g,x,y$ 
and relations \eqref{eq:relations-H}, \eqref{relations-Dgamma}, \eqref{relations-H-dual} and
\begin{align}\label{relation-DH}
\begin{aligned}
\zeta y &= y \zeta + y, & \zeta x&=x \zeta + x, & v g &=g v + g u,\\ 
u g &=g u , &  v x &= x v + (1-g) + x u, & u x &= x u, \\
u y &=y u +(1-g) &  v y &=y v -g \zeta + y u.
\end{aligned} 
\end{align} 
The comultiplication and antipode are given by \eqref{eq:hopf-structure-double-H-1}, 
\eqref{eq:hopf-structure-double-H-2} and \eqref{eq:hopf-structure-double-H-3}. The following family 
is a a PBW-basis of $D(H)$:
\begin{align*}
\left\{ x^n\,y^r\,g^m\,\zeta^k u^i v^j: \ i,j,k,n,r\in\I_{0, p-1}, \ m\in\fp \right\}
\end{align*}
\end{prop}
\begin{proof}
Let $A$ be an algebra presented by generators $U=\{ u, v,\zeta,g,x,y\}$ and the relations above. 
These relations hold in $D(H)$, thus $A\twoheadrightarrow D(H)$, and $\dim A \geq p^6$. We claim that $B = (x^n\,y^r\,g^m\,\zeta^k u^i v^j)_{\substack{i,j,k,n,r\in\I_{0, p-1} \\ m\in\fp}}$ generates $A$. 
We define a total order in $U$ by declaring $x<y<g<\zeta< u< v$.
By the defining relations, any product $ab$ with $a,b\in U$ and $b < a$ can be written as a linear combination of monomials $c_1\cdots c_s$ with $c_1\leq c_2\leq\cdots\leq c_s\in U$. The claim follows from this together with the relations
$x^p = 0$, $y^p = 0$, $u^p = 0$, $v^p = 0$, $g^p = 1$, $\zeta^p = \zeta$.
Hence $A\simeq D(H)$. 
\end{proof}

\subsection{An exact sequence}
We recall first the definition of short
exact sequence of Hopf algebras see e. g. \cite{ad, Sch,Hf}.

\begin{definition}
A sequence of morphisms of Hopf algebras
\begin{align*}
A\xhookrightarrow[]{\iota} C \xrightarrowdbl[]{\pi} B
\end{align*}
is exact if the following conditions holds:
\begin{multicols}{2}
\begin{enumerate}[leftmargin=*,label=\rm{(\roman*)}]
\item\label{suc-exacta-1} $\iota$ is injective.
\item\label{suc-exacta-2} $\pi$ is surjective.
\item\label{suc-exacta-3} $\ker\pi = C\iota(A)^+$.
\item\label{suc-exacta-4} $\iota(A) = C^{\operatorname{co} \pi}$.
\end{enumerate}
\end{multicols}
\end{definition}

\begin{remark}\label{remark-exact-sequence-hopf}
If $A\xhookrightarrow[]{\iota} C$ is faithfully flat and $\iota(A)$ is stable by the left adjoint action of $C$ then
\ref{suc-exacta-1}, \ref{suc-exacta-2} and \ref{suc-exacta-3} imply \ref{suc-exacta-4}, see \cite[1.2.5, 1.2.14]{ad}, \cite{Sch}.
\end{remark}

Let $\{h,e,f\}$ be the Cartan generators of $\spl_2(\ku)$.

\begin{prop}\label{prop:double-H-ext} The subalgebra $\nucleo$ of $D(H)$ generated by
$g$, $x$ and $u$ is a normal local commutative Hopf subalgebra of $D(H)$ of dimension $p^3$ with defining relations
\begin{align}\label{eq:rels-subalgebra}
g^p &= 1, & x^p &= 0, & u^p &= 0. 
\end{align}
It gives rise to the exact sequence of Hopf algebras 
\begin{align*}
\xymatrix{ \nucleo \ar@{^{(}->}[r]  & D(H) \ar@{->>}[r]  & \mathfrak u(\spl_2(\ku)).}
\end{align*}
\end{prop}

Since $\nucleo $ is commutative and $\mathfrak u(\spl_2(\ku))$ is cocommutative, $D(H)$ arises as an abelian extension.

\pf By Proposition \ref{prop:double-H} $\nucleo $ is a commutative Hopf subalgebra of dimension $p^3$ and there is a surjective algebra map from the commutative algebra presented by relations \eqref{eq:rels-subalgebra} to $\nucleo $. By dimension counting this map is an isomorphism. 
By inspection $g$, $x$ and $u$ are normal, hence so is $\nucleo $.
Thus $D(H)\nucleo ^+$ is a Hopf ideal of $D(H)$ and the quotient $D(H)/D(H)\nucleo ^+$ is isomorphic to $\mathfrak u(\spl_2(\ku))$ via $\zeta\mapsto h$, $y\mapsto \frac{1}{2}e$ and $v\mapsto f$. \epf

\begin{remark}\label{rem:algebraic-group-normal-hopfsubalgebra}
Let $\Gb = (\Gb_a \times \Gb_a) \rtimes \Gb_m$ be the semidirect product where $\Gb_m$ acts on $\Gb_a \times \Gb_a$
by 
$\lambda_ \cdot(t_1, t_2) =(t_1, \lambda t_2)$, $\lambda \in \kut$, $t_1, t_2\in\ku$.
Then the algebra of regular functions $\cO(\Gb)$ is isomorphic to $\ku[X_1,X_2, T^{\pm 1}]$
and there is a short exact sequence of Hopf algebras
\begin{align*}
\xymatrix{ \cO(\Gb)  \ar@{^{(}->}[rr] ^{\operatorname{Fr}} & & \cO(\Gb) \ar@{->>}[rr] ^{\pi} & & \nucleo },
\end{align*}
 with $\pi\colon\cO(\Gb)\longrightarrow \nucleo $
given by $T\mapsto g$, $X_1\mapsto u$, $X_2\mapsto x$. In other words $\operatorname{Spec} \nucleo $ is the kernel of the Frobenius
endomorphism of $\Gb$.
\end{remark}

\begin{theorem}\label{th:irrep-restricted-jordan}
 There are exactly $p$ isomorphism classes of simple $D(H)$-modules which have dimensions $1,2,\dots,p$. 
\end{theorem}

\pf The two-sided ideal $D(H)\nucleo ^+$, generated by $x$, $u$ and $g-1$
is nilpotent, hence contained in the Jacobson radical of $D(H)$ and $\Irr D(H) \simeq \Irr \mathfrak u(\spl_2(\ku))$.
Then the well-known classification of the latter appplies.
\epf

\subsection{Simple modules}
Here we describe the simple modules of $D(H)$
as quotients of Verma modules reproving Theorem \ref{th:irrep-restricted-jordan}. 

First $D(H) = \oplus_{n\in\Z}D(H)^n$ is $\Z$-graded by 
\begin{align*}
\deg x =\deg y = -1, && \deg u =\deg v = 1, && \deg g = \deg \zeta = 0.
\end{align*}
Thanks to the PBW-basis, the multiplication induces a linear isomorphism
\begin{align*}
\toba(V)\ot D(\ku\Gamma)\ot\toba(W)\longrightarrow D(H)
\end{align*}
called the triangular decomposition of $D(H)$. The subalgebras 
\begin{align*}
\toba(W) &\eqqcolon\D^{> 0},& \toba(V)&\eqqcolon\D^{< 0}& \text{ and } & D(\ku\Gamma)
\end{align*}
are graded and satisfy

\begin{enumerate}
\smallbreak\item $\D^{> 0}\subseteq\oplus_{n\in\N_0} D(H)^n$, $\D^{< 0}\subseteq\oplus_{n\in -\N_0} D(H)^n$ and $D(\ku\Gamma)\subseteq D(H)^0$.

\smallbreak
\item $(\D^{> 0})^0 = \ku = (\D^{< 0})^0$.

\smallbreak
\item $\D^{\geq0} := D(\ku\Gamma)\D^{> 0}$ and $\D^{\leq0} := \D^{< 0}D(\ku\Gamma)$ are subalgebras of $D(H)$.
\end{enumerate}

In this context the simple modules of $D(H)$ arise inducing from $\D^{\geq0}$. 
The elements of $\Lambda := \Irr D(\ku\Gamma)$ are called \emph{weights}. 
Since $\D^{> 0}$ is local, the (homogeneous)
projection $\D^{\geq0} \twoheadrightarrow D(\ku\Gamma)$ allows to identify
$\Lambda \simeq \Irr \D^{\geq0}$.
The  Verma module associated to $\lambda\in\Lambda$ is
\begin{align*}
\verma(\lambda) = \Ind^{D(H)}_{\D^{\geq0}}\lambda = D(H)\ot_{\D^{\geq0}} \lambda.
\end{align*}
By a standard argument, $\verma(\lambda)$ is indecomposable. 
Let $L(\lambda)$ be the head of $\verma(\lambda)$. 
The following result is well-known, see for instance \cite[Theorem 2.1]{Vay}.

\begin{lemma} The map $\lambda \mapsto L(\lambda)$ gives a bijection $\Lambda \simeq \Irr D(H)$.  \qed
\end{lemma}

The set $\Lambda$ is easy to compute since $D(\ku\Gamma) \simeq \ku \Gamma \otimes \ku^{\Gamma}$
and $\ku \Gamma$ is local. Given $k \in \fp$, let $\lambda_k = \ku w_k$ be the one-dimensional vector space with action
\begin{align*}
g\cdot w_k = w_k, && \zeta\cdot w_k = k w_k.
\end{align*}

\begin{lemma}The map $k \mapsto \lambda_k$ provides a bijection $\fp \simeq \Lambda$.
\qed.
\end{lemma}

We fix $k \in \fp$ and compute $L(\lambda_k)$. 
Since $\verma(\lambda)$ is free as a $\D^{< 0}$-module with basis $( w_k)$, $( w_k^{(i,j)})_{i,j\in\I_{0, p-1}}$ is a linear basis of $\verma(\lambda)$, with $ w_k^{(i,j)} = x^i y^j \cdot w_k$.
This makes $\verma(\lambda_k)$ a graded module by $\deg w_k^{(i,j)} = \deg (x^i y^j) = -i-j$, so $\verma(\lambda_k) =\oplus_{n\leq 0}\verma(\lambda_k)_{n} $
and $\verma(\lambda_k)_0 = \ku w_k$. Any proper submodule of $\verma(\lambda_k)$ is necessarily contained in $\oplus_{n\leq -1}\verma(\lambda_k)_{n}$, so the sum of all proper submodules is proper, and $\verma(\lambda_k)$ has an unique simple quotient $L(\lambda_k)$. We divide $\verma(\lambda_k)$ by proper submodules until we get a simple one.
\begin{lemma}
The submodule $N_k$ of $\verma(\lambda_k)$ generated by $ w_k^{(1,0)}$ is proper.
\end{lemma}
\begin{proof}
The action of $u$ and $ v$ gives
\begin{align*}
u x\cdot w_k = x u\cdot w_k = 0, && v x\cdot w_k = x v \cdot w_k + (1-g)\cdot w_k + x u\cdot w_k= 0.
\end{align*}
So $\D^{> 0}\cdot w_k^{(1,0)} = 0$. Then $N_k = \D^{\leq0}\cdot w_k^{(1,0)} \subseteq \oplus_{n\leq -1}\verma(\lambda_k)_{n}$
 is proper. 
\end{proof}
Let $V_k = \verma(\lambda_k) / N_k$ and let $y_j$ be the class of $w_k^{(0,j)}$ in $V_k$.
\begin{lemma}
The family $(y_j)_{j\in\I_{0, p-1}}$ generates linearly $V_k$ and $g$, $ u$ and $x$ act trivially on $V_k$.
\end{lemma}
\begin{proof}
First we claim that the class of $w_k^{(i,j)} = 0$ in $V_k$ if $i\neq 0$:
it is enough to show that $w_k^{(1,j)} = 0$ in $V_k$ what follows by induction on $j$ using \eqref{eq:relations-H}.
Then $x$ acts trivially on $V_k$ and $(y_j)_{j\in\I_{0, p-1}}$ generates linearly $V_k$. 
Also $g$ and $ u$ act trivially on these generators by induction on $j$ using \eqref{relation-DH}.
\end{proof}
The action of $D(H)$ on $V_k$ can be computed inductively:
\begin{align}\label{eq:action-Lk}
\begin{aligned}
\zeta\cdot y_j &= (k+j)y_j, & v\cdot y_j &= \frac{1}{2}j(1-2k-j) y_{j-1}, & y\cdot y_j &= y_{j+1},\\
x\cdot y_j &= 0, & u\cdot y_j &= 0, & g\cdot y_j &= y_j.
\end{aligned}
\end{align}
Let $r\in \I_{0, p-1}$ be the representative of $-2k$, i.e. $r \equiv -2k\mod p$.
Then $\widetilde{V}_k := D(H)y_{r+1}$ is a proper submodule of $V_k$ because $\D^{> 0}\cdot y_{r+1} = 0$.

\begin{prop}\label{prop:irrep-D(H)}
The module $L_k = V_k/ \widetilde{V}_k$ is simple of dimension $r+1$.
\end{prop}
It follows that $L_k = L(\lambda_k)$, the head of the Verma module $M(\lambda_k)$.

\begin{proof} Let $z_j$ be the class of $y_j$ in $L_k$; the action of $D(H)$ on the $z_j$'s is still given by \eqref{eq:action-Lk}.
Then $(z_j)_{j\in\I_{0, r}}$ is a basis of $L_k$. 
 To see that $L_k$ is simple, we show that every $0 \neq z \in L_k$ generates $L_k$. Let $z = \sum_{j=0}^{m} c_j\,z_j$ 
 with $m\leq r$ and $c_m\neq 0$. Then $ v^m\cdot z\in \kut z_0$, and $D(H)\cdot z = L_k$.
\end{proof}

\section{The double of the Jordan plane}
\subsection{Definition and basics}
Here we define an infinite-dimensional Hopf algebra $\widetilde{D}$ such that $\widetilde{D}\twoheadrightarrow D(H)$. 
We first consider two Hopf algebras $\widetilde{H}$ and $\widetilde{K}$ such that $\widetilde{H}\twoheadrightarrow H$ 
and $\widetilde{K}\twoheadrightarrow K \coloneqq (H^*)^{\operatorname{op}}$. Then $\widetilde{D} \coloneqq \widetilde{H} \bowtie_\sigma \widetilde{K}$ for a suitable 2-cocycle $\sigma$ as in \cite{dt}. The Hopf algebra $\widetilde{H}$ was studied in \cite{abff} (assuming $\car \ku =0$).

Let $\widetilde{\Gamma} = \langle \gen \rangle \simeq \Z$. 
Let $\ku[\zeta]$ be the polynomial algebra, a Hopf algebra with $\zeta$ primitive. 
We have Hopf algebra maps $\ku\widetilde{\Gamma}\twoheadrightarrow\ku\Gamma$ and $\ku[\zeta]\twoheadrightarrow\ku^\Gamma$. 
By the same formulas as in $\eqref{Eq:V_as_YD_module}$ 
and $\eqref{Eq:W_as_YD_module}$, $V\in\ydtildeG$ and $W \in \ydtildeGd$. Thus we have the following braided Hopf algebras and their bosonizations
\begin{align*}
\wtoba &= T(V)/(yx-xy+\frac{1}{2}x^2) \in \ydtildeG,& \widetilde{H} &= \wtoba\#\ku\widetilde{\Gamma}; \\
\ttoba &= T(W)/( v u- u v-\frac{1}{2} u^2)\in \ydtildeGd,& \widetilde{K} &= ( \ttoba\#\ku[\zeta])^{\operatorname{op}}.
\end{align*}

The algebras $\wtoba$ and $\ttoba$ are isomorphic to the well-known Jordan plane; but they are not isomorphic as coalgebras. 

\begin{lemma} The families $(x^iy^j\gen^{k})_{\substack{i,j\in\N_0 \\ k\in \Z }}$ and $(\zeta^i u^j v^k)_{i,j,k\in\N_0}$ are PBW-bases of $\widetilde{H}$ and $\widetilde{K}$ respectively. 
The algebra $\widetilde{H}$ is presented by generators $x,y,\gen^{\pm}$ and relations
\begin{align}\label{Relaciones-tilde-H}
\begin{aligned}
\gen x &= x\gen , & \gen y &= y\gen +x\gen , & yx &= xy - \frac{1}{2} x^2, & \gen ^{\pm} \gen ^{\mp} &= 1.\\
\end{aligned}
\end{align}
The algebra $\widetilde{K}$ is presented by generators $ u, v,\zeta$ and relations
\begin{align}\label{eq:Relaciones-tilde-H-estrella}
v\zeta = \zeta v + v, &&  u\zeta= \zeta u + u, && v u = u v-\frac{1}{2} u^2.
\end{align}
The coproducts of $\widetilde{H}$ and $\widetilde{K}$ are determined by \eqref{eq:hopf-structure-double-H-1}, respectively \eqref{eq:hopf-structure-double-H-3}.
\end{lemma}

\begin{proof} It is well-known that $(x^iy^j)_{i,j\in\N_0}$, $( v^i u^j)_{i,j\in\N_0}$ are PBW-bases of $\wtoba$ and $\ttoba$ respectively.
Thus the first claim follows (notice the change of order in the monomials of $\widetilde{K}$).
Let $A$ and $B$ be the algebras presented by relations \eqref{Relaciones-tilde-H} and \eqref{eq:Relaciones-tilde-H-estrella} respectively. 
The relations are valid in $\widetilde{H}$ and $\widetilde{K}$, thus we have epimorphisms of Hopf algebras $A\twoheadrightarrow\widetilde{H}$ and $B\twoheadrightarrow\widetilde{K}$. 
Clearly $(\gen ^{i} x^j y^k)_{\substack{j,k\in\N_0 \\ i\in \Z }}$ and $( u^i v^j\zeta^k)_{i,j,k\in\N_0}$ are
linear generators of $A$ and $B$ respectively. Since their images by the
morphisms are linear independent, the result follows.
\end{proof}

To define $\widetilde{D}$, we need a skew-pairing $\tau$ between $\widetilde{H}$ and
$\widetilde{K}$ as in \cite{dt} i.e. a linear map $\tau\colon\widetilde{H}\otimes\widetilde{K}\longrightarrow\ku$ satisfying
\begin{align}\label{eq:skew-pairing}
\begin{aligned}
\tau(h\widetilde h\ot k) &= \tau(h \ot k_{(1)}) \tau(\widetilde h \ot k_{(2)}),&
\tau(1 \ot k) &= \varepsilon(k), & h,\widetilde h &\in\widetilde{H}\\
 \tau(h \ot \widetilde kk) &= \tau(h_{(1)} \ot k) \tau(h_{(2)} \ot \widetilde k),&
\tau(h \ot 1) &= \varepsilon(h), &k, \widetilde k &\in \widetilde{K}.
\end{aligned}
\end{align}
Then $\tau$ is convolution invertible with inverse $\tau^{-1}(b,c) = \tau(\Ss(b),c)$. 
By \eqref{eq:skew-pairing}, a skew-pairing $\tau$ is equivalent to a Hopf algebra homomorphism
$\varphi$ from $\widetilde{H}^{\operatorname{cop}}$ to the Sweedler dual ${\widetilde{K}}^{\circ}$ of $\widetilde{K}$. 

\begin{lemma}
There exists a unique skew-pairing $\tau\colon\widetilde{H}\otimes\widetilde{K}\longrightarrow\ku$
such that
\begin{align*}
\tau(x \ot u) &= 0, & \tau(y\ot u) &= 1, & \tau(\gen ^{\pm 1}\ot u) &= 0,\\
\tau(x\ot v) &= 1, & \tau(y\ot v) &= 0, & \tau(\gen ^{\pm 1}\ot v) &=0,\\
\tau(x\ot\zeta) &=0, & \tau(y\ot\zeta) &=0, & \tau(\gen ^{\pm 1}\ot\zeta) &= \pm 1. 
\end{align*}
\end{lemma}
\pf Left to the reader.
\epf

Let $A = \widetilde{H}\ot\widetilde{K}$ with the structure of tensor product Hopf algebra.
The 2-cocycle $\sigma\colon A\ot A\longrightarrow A\ot A$ associated to $\tau$ is given by
\begin{align*}
\sigma(a\ot b, c\ot d) &= \varepsilon(a)\varepsilon(d) \tau(c\ot b), & a,b,c,d &\in A.
\end{align*}
 We define $\widetilde{D}$ as the
Hopf algebra $A$ twisted by $\sigma$, i.e $\widetilde{D}=A_\sigma = \widetilde{H} \bowtie_\sigma \widetilde{K}$.

\begin{prop}\label{prop:Dtilde} \begin{enumerate}[leftmargin=*,label=\rm{(\roman*)}] 
\item\label{prop:Dtilde-dos} The algebra $\widetilde{D}$ is presented by generators $u$, $v$, $\zeta$, $\gen ^{\pm 1}$, $x$, $y$ with relations \eqref{Relaciones-tilde-H}, \eqref{eq:Relaciones-tilde-H-estrella} and
\begin{align}\label{eq:relation-tilde-D}
\begin{aligned}
\zeta y &= y \zeta + y, & \zeta x&=x \zeta + x, & v \gen &=\gen v + \gen u,\\ 
u \gen &=\gen u , &  v x &= x v + (1-\gen ) + x u, & u x &= x u, \\
u y &=y u +(1-\gen) &  v y &=y v -\gen \zeta + y u, & \zeta \gen &= \gen \zeta.
\end{aligned} 
\end{align}
The coproduct is determined by \eqref{eq:hopf-structure-double-H-1} and \eqref{eq:hopf-structure-double-H-3}.

\smallbreak
\item\label{prop:Dtilde-uno} The following family is a PBW-basis of $\widetilde{D}$:
\begin{align}\label{eq:pbw-Dtilde}
\left\{x^n\,y^r\,\gen ^m\,\zeta^k u^i v^j: \, i,j,k,n,r\in\N_0, \ m\in\Z \right\}.
\end{align}

\smallbreak\item\label{prop:Dtilde-tres} There exists a Hopf algebra epimorphism $\widetilde{D}\twoheadrightarrow D(H)$. 

\end{enumerate}
\end{prop}
\begin{proof}
\ref{prop:Dtilde-uno} follows from the definition.
\ref{prop:Dtilde-dos}: Let $A$ be the algebra presented as above. By construction, the relations are valid in $\widetilde{D}$ and thus we have an epimorphism
$A\twoheadrightarrow \widetilde{D}$. 
Using the commutation relations we see that $A$ is linearly generated by \eqref{eq:pbw-Dtilde}.
By \ref{prop:Dtilde-uno}, $A\simeq \widetilde{D}$.

\ref{prop:Dtilde-tres} follows from \ref{prop:Dtilde-dos} since the relations are satisfied in $D(H)$. 
\end{proof}

\begin{remark} If we assume that $\car \ku = 0$, then
the definition of $\widetilde{D}$ makes sense  and items \ref{prop:Dtilde-dos} and \ref{prop:Dtilde-uno} of
Proposition \ref{prop:Dtilde} are also valid.
\end{remark}

We next show $\widetilde{D}$ is in fact an extension of $D(H)$ by a central Hopf subalgebra $Z$. 
We shall need the following lemma.

\begin{lemma}\label{Commutation-relations-D-tilde}
The following commutation relations  hold in $\widetilde{D}$ 
for  $n, m \in \N$:
\begin{align*}
&\begin{aligned}
\zeta^n x^m &= \sum_{\ell=0}^{n} \binom{n}{\ell}m^{n-\ell} x^m \zeta^\ell, &
\zeta^n y^m &= \sum_{\ell=0}^{n} \binom{n}{\ell}m^{n-\ell} y^m \zeta^\ell,\\
v^m \zeta^n &= \sum_{\ell=0}^{n} \binom{n}{\ell}m^{n-\ell} \zeta^\ell v^m, &
u^m \zeta^n &= \sum_{\ell=0}^{n} \binom{n}{\ell}m^{n-\ell} \zeta^\ell u^m,
\end{aligned}
\\
\gen ^n y^\ell &= \sum_{k=0}^{\ell} \binom{\ell}{k} (-1)^k\frac{[-2n]^{[k]}}{2^k} x^k y^{\ell-k} \gen ^n,\\
y^\ell x^n &= \sum_{k=0}^{\ell} \binom{\ell}{k} (-1)^k \frac{[n]^{[k]}}{2^k} x^{n+k} y^{\ell-k},\\
v^\ell u^n &= \sum_{k=0}^{\ell} \binom{\ell}{k} (-1)^k \frac{[n]^{[k]}}{2^k} u^{n+k} v^{\ell-k},\\
v^\ell \gen ^n &= \sum_{k=0}^{\ell} \binom{\ell}{k} (-1)^k\frac{[-2n]^{[k]}}{2^k} \gen ^n u^k v^{\ell-k},
\\
v x^n &= x^n v + n x^{n-1}(1-\gen ) + n x^n u,\\
v^n x &= x v^n + n x u v^{n-1} + \frac{n(n-1)}{4} x u^2 v^{n-2} + n v^{n-1}
+ \frac{n(n-1)}{2} u v^{n-2} \\ &- n \gen v^{n-1} - n(n-1) \gen u v^{n-2} 
-\frac{1}{4} n(n-1)(n-2) \gen u^2 v^{n-3},\\
v^n y &= y v^n + n y u v^{n-1} - \frac{n(n-1)}{4} y u v^{n-2} - n \gen \zeta v^{n-1} -
n(n-1) \gen \zeta u v^{n-2} \\ & -\frac{1}{4} n(n-1)(n-2) \gen \zeta u^2 v^{n-3} -
\frac{n(n-1)}{2} \gen v^{n-1} - \frac{n(n-1)^2}{2} \gen u v^{n-2} \\ &-
\frac{1}{8} n(n-1)^2 (n-2) \gen u^2 v^{n-3},\\
u y^n &= y^n u + n y^{n-1} - \sum_{k=0}^{n-1} \binom{n}{k+1} \frac{(k+1)!}{2^k} y^{n-1-k} x^k \gen ,\\
v y^n &= y^n v + n y^n u + \frac{n(n-1)}{2} y^{n-1} -
\sum_{k=0}^{n-1}\binom{n}{k+1} \frac{(k+1)!}{2^k} y^{n-1-k} x^k \gen \zeta \\ & -
\sum_{k=0}^{n-1} \binom{n}{k+1} \frac{(n-1)(k+1)!}{2^k} y^{n-1-k} x^k \gen ,\\
u^n y &= y u^n + n (1-\gen) u^{n-1},
\end{align*}
The comultiplication satisfies
\begin{align*}
\Delta(x^n) &= \sum_{k=0}^n \binom{n}{k} x^{n-k} \gen ^k \ot x^k,\\
\Delta(y^n) &= \sum_{k=0}^{n}\sum_{i=0}^{k}\binom{n}{k}\binom{k}{i} (-1)^i\frac{[n-k]^{[i]}}{2^i} 
\gen ^{n-k} x^i y^{k-i} \ot y^{n-k}.
\end{align*}
\end{lemma}
\begin{proof} Straightforward by induction. \end{proof}

For our next statement we need to set up the notation. 
Let
\begin{align}\label{eq:alggroup-B}
\Bb = \left((\Gb_a \times \Gb_a) \rtimes \Gb_m\right) \times \Hb_3
\end{align}
 be the algebraic group that in the first factor has the semidirect product where $\Gb_m$ acts on $\Gb_a \times \Gb_a$
by 
$\lambda_ \cdot(r_1, r_2) =(\lambda r_1, \lambda r_2)$, $\lambda \in \kut$, $r_1, r_2\in\ku$ while in the second factor
appears the Heisenberg group $\Hb_3$ i.e. the group of upper triangular matrices with ones in the diagonal.
(The first factor is not the same as the group in Remark \ref{rem:algebraic-group-normal-hopfsubalgebra}).

\begin{prop}\label{prop-D-tilde-extension}
Let $Z$ be the subalgebra of $\widetilde{D}$ generated by the elements $\gen ^p$, $x^p$, $y^p$, $u^p$, $v^p$ and $\zeta^{(p)} \coloneqq \prod_{i=1}^{p} (\zeta +i -1) = \zeta^p - \zeta$.
Then 
\begin{enumerate}[leftmargin=*,label=\rm{(\roman*)}]
\item\label{Z-1} $Z$ is a central Hopf subalgebra of $\widetilde{D}$.

\smallbreak
\item\label{Z-2} $\widetilde{D}$ is a finitely generated free $Z$-module.

\smallbreak
\item\label{Z-3} $Z\xhookrightarrow[]{\iota} \widetilde{D} \xrightarrowdbl[]{\pi} D(H)$ is short exact sequence of Hopf algebras.

\smallbreak
\item\label{Z-4} $Z\simeq \ku[T^{\pm}]\ot\ku[X_1,\dots,X_5]$ as an algebra. In particular $Z$ is a domain.

\smallbreak
\item\label{Z-5} $Z\simeq \cO(\Bb)$ as Hopf algebras.
\end{enumerate}
\end{prop}

\begin{proof}
\ref{Z-1} By Lemma \ref{Commutation-relations-D-tilde} $Z$ is a central subalgebra of $\widetilde{D}$;
$u^p,\zeta^{(p)}\in\Pc(\widetilde{D})$ and $\gen ^p\in G(\widetilde{D})$. Also $\Delta(x^p) = x^p\ot 1 + \gen ^p\ot x^p$ and
$\Delta(y^p) = y^p\ot 1 + \gen ^p\ot y^p$.
It only remains to show that
\begin{align}\label{comultiplication-v-p}
\Delta(v^p) = v^p \ot 1 + 1 \ot v^p +\zeta^{(p)}\ot u^p.
\end{align} 
Since $v\in\widetilde{K} = ( \ttoba\#\ku[\zeta])^{\operatorname{op}}$ it is enough to know the coaction $\delta(v^p)$ and the braided comultiplication of $v^p$
in $\ttoba$. It can be proved inductively that
\begin{align*}
\Delta_{\ttoba}(v^n) &= v^n\ot 1 + 1\ot v^n + \sum_{k=1}^{n-1} 
\sum_{i=0}^{k}\binom{n}{k}\binom{k}{i} \frac{[n-k]^{[i]}}{2^i}  v^{n-k}\ot u^i v^{k-i},\\
\delta(v^n) &= 1\ot v^n + \zeta^n \ot u^n + \sum_{k=1}^{n-1} \sum_{t=k}^{n}\binom{n}{t}\stirling{t}{k} \frac{1}{2^{t-k}}
\zeta^k\ot u^t v^{n-t},
\end{align*}
 for $n\in\N$.
Since $\zeta^{(p)} = \sum_{k=0}^{p} \stirling{p}{k} \zeta^k = \zeta^p - \zeta$, we have 
\begin{align}\label{eq:stirling}
\stirling{p}{k} &= 0, \ k=2,\dots,p-1, &
\stirling{p}{p}&=1, &\text{ and }\stirling{p}{1} &= -1.
\end{align} 
Then $\delta(v^p) = 1\ot v^p + \zeta^{(p)}\ot u^p$,  
$\Delta_{\ttoba}(v^p) = v^p\ot 1 + 1\ot v^p$, and  \eqref{comultiplication-v-p} follows.

\smallbreak
\ref{Z-2} To prove this we consider another basis of $\widetilde{D}$ using a different basis of $\ku[\zeta]$. The family of
polynomials 
\begin{align*}
\left\{(\zeta^{(p)})^k \zeta^j: \ k\in\N_0, \, j\in\I_{0, p-1} \right\}
\end{align*}
is a basis of $\ku[\zeta]$. Indeed, these polynomials are linearly independent since
they have different degrees. We prove the polynomials $\zeta^n$ can be written as linear combinations of them by induction on $n$.
The case $n=0$ is obvious. If $\zeta^n = \sum_{k=0}^{n}\sum_{j=0}^{p-1} a_{k,j} (\zeta^{(p)})^k \zeta^j$ for some 
constants $a_{k,j}$ then
\begin{align*}
\zeta^{n+1}& = \sum_{k=0}^{n}\sum_{j=0}^{p-1} a_{k,j} (\zeta^{(p)})^k \zeta^{j+1} =\\
&=
\sum_{k=0}^{n}\sum_{j=1}^{p-1} a_{k,j-1} (\zeta^{(p)})^k \zeta^{j} + \sum_{k=0}^{n} a_{k,p-1} (\zeta^{(p)})^{k+1}
+ \sum_{k=0}^{n} a_{k,p-1} (\zeta^{(p)})^{k} \zeta.
\end{align*}
By Proposition \ref{prop:Dtilde} \ref{prop:Dtilde-uno}, the following is a linear basis of $\widetilde{D}$:
\begin{align*}
\left\{x^n\,y^r\,\gen ^m\,(\zeta^{(p)})^k \zeta^\ell u^i v^j : \ i,j,k,n,r\in\N_0, \ m\in\Z, \ \ell\in\I_{0, p-1} \right\}
\end{align*} 
Hence the following is a basis of $\widetilde{D}$ as $Z$-module:
\begin{align*}
\left\{x^n\,y^r\,\gen ^m\, \zeta^\ell u^i v^j : \ i,j,\ell ,n,r, m\in\I_{0, p-1} \right\}.
\end{align*} 

\smallbreak
\ref{Z-3} By Propositions \ref{prop:Dtilde} and \ref{prop:double-H}, $\ker \pi$ is the ideal generated by
\begin{align*}
&x^p, & &y^p,& &\gen ^p-1,& &u^p,& &v^p& &\text{ and } &\zeta^{(p)}.
\end{align*}
 This is clearly equal to $\widetilde{D}\iota(Z)^+$. Then
Remark \ref{remark-exact-sequence-hopf}, \ref{Z-1} and \ref{Z-2} apply.

\smallbreak
\ref{Z-4} By the proof of \ref{Z-2}, the following is a basis of $Z$:
\begin{align*}
\left\{x^{pn}\,y^{pr}\,\gen ^{pm}\, (\zeta^{(p)})^k u^{pi} v^{pj} : \ i,j,k ,n,r\in\N_0 ,\, m\in\Z \right\}.
\end{align*} 
Hence the following assignement
\begin{align*}
T &\mapsto \gen ^p,& X_1&\mapsto x^p,& X_2&\mapsto y^p, &
X_3&\mapsto \zeta^{(p)},& X_4&\mapsto u^p, &X_5 &\mapsto v^p.
\end{align*}
 provides an algebra isomorphism
$Z\simeq \ku[T^{\pm}]\ot\ku[X_1,\dots,X_5]$.

\smallbreak
\ref{Z-5} By \ref{Z-4} $Z$ is a commutative Hopf algebra without nilpotent elements. Then $Z\simeq \cO(G)$ for the
algebraic group $G=\Alg(Z,\ku) \simeq \ku^\times\times\ku^5$. The multiplication of $G$ is induced from the
comultiplication of $Z$ and is given for $\omega = (\lambda,r_1,r_2,t_1,t_2,t_3)$, $\omega' = (\lambda',r_1',r_2',t_1',t_2',t_3')\in \ku^\times\times\ku^5$ by
\begin{align}\label{Realization_of_B}
\omega\cdot\omega' =(\lambda \lambda',r_1+\lambda r_1',r_2+\lambda r_2',t_1+t_1',t_2+t_2', t_3+t_3'+t_1 t_2').
\end{align}
Clearly $G=\Bb$ as claimed.
\end{proof}

\subsection{Another exact sequence} 
In this Subsection we start assuming that $\car \ku \neq 2$, i.e. $\car \ku = 0$ is also allowed.
Let $\{h,e,f\}$ be the Cartan generators of $\spl_2(\ku)$. 
Let $\pi\colon \widetilde{D}\longrightarrow U(\spl_2(\ku))$ be the Hopf algebra map given by
 $x\mapsto 0$, $y\mapsto \frac{1}{2} e$, $u\mapsto 0$, $v\mapsto f$, $\gen\mapsto 1$ and $\zeta \mapsto h$. 
 Let $C$ be the subalgebra of $\widetilde{D}$ generated by $x$, $u$ and $\gen$; clearly it is a normal Hopf subalgebra. 
 Let $\Gb$ be the algebraic group defined in Remark \ref{rem:algebraic-group-normal-hopfsubalgebra}.

 \begin{prop}\label{prop:another-exseq}
 \begin{enumerate}[leftmargin=*,label=\rm{(\roman*)}]
 \item\label{D-1} $C\simeq \cO(\Gb)$ as Hopf algebras.
 \item\label{D-2} There is an exact sequence of Hopf algebras 
 \begin{align*}
 \cO(\Gb) \hookrightarrow \widetilde{D} \xrightarrow{\pi} U(\spl_2(\ku)).
 \end{align*}
 \end{enumerate}
 \end{prop}
 \begin{proof}
 \ref{D-1} follows from the definition using the PBW-basis.

 \ref{D-2} By the normality of $\cO(G)$, $\ker \pi = \left\langle x, u, \gen-1 \right\rangle = \widetilde{D} \cO(G)^+$. 
 Then Remark \ref{remark-exact-sequence-hopf} applies: 
 $\cO(G)$ is stable by the adjoint action and $\widetilde{D}$ is a free module over $\cO(G)$ by the PBW-basis, so the inclusion is faithfully flat.
 \end{proof}
 
 We come back to the assumption $\car \ku > 2$. Recall the commutative diagram
 \begin{align}\tag{\ref{eq:diagram-exact sequences}}
 \begin{aligned}
 \xymatrix{  \cO(\Gb)  \ar@{^{(}->}[r] \ar@{^{(}->}[d]_{\operatorname{Fr}} & \cO(\Bb) \ar@{^{(}->}[d] \ar@{->>}[r] & \cO(\Gb_a^3) \ar@{^{(}->}[d]
 \\
  \cO(\Gb)  \ar@{^{(}->}[r] \ar@{->>}[d] & \widetilde{D} \ar@{->>}[r] \ar@{->>}[d] & U(\spl_2(\ku)) \ar@{->>}[d]
 \\
  \nucleo \ar@{^{(}->}[r] & D(H) \ar@{->>}[r] & \mathfrak u(\spl_2(\ku)).
 }
 \end{aligned}
 \end{align}
 
\begin{prop}\label{prop:diagram-exact sequences}
All columns and rows in \eqref{eq:diagram-exact sequences} are exact sequences. 
\end{prop}
 
 \begin{proof}
The middle row is exact by the previous proposition and the bottom by Proposition \ref{prop:double-H-ext}.
For the rightmost column we need to prove that the Hopf subalgebra  $Z' = \left\langle e^p, f^p, h^p-h \right\rangle$ of $U(\spl_2(\ku))$ is $\cO(\Gb_a^3)$. Taking the
basis $B=\{f^i (h^p-h)^j h^k e^\ell\colon i,j,\ell\in \N_0, k\in\I_{0, p-1}\}$ of $U(\spl_2(\ku))$ we 
see that the assignment $X_1\mapsto f^p$, $X_2 \mapsto h^p - h$ and $X_3\mapsto e^p$  gives an algebra isomorphism 
$Z'\simeq \ku[X_1,X_2,X_3] \simeq \cO(\Gb_a^3)$. Comparing comultiplications, the previous isomorphism is of Hopf algebras. $\cO(\Gb_a^3)$ is stable 
by the adjoint action of  $U(\spl_2(\ku))$ and a free module over $\cO(\Gb_a^3)$ using the basis $B$, then Remark \ref{remark-exact-sequence-hopf} applies and the column is exact.

 We next describe
explicitly the top arrow. $\phi\colon \cO(\Gb) \to \cO(\Bb)$ is given by $x\mapsto x^p$, $u\mapsto u^p$ and $\gen \mapsto \gen^p$, and
$\psi \colon \cO(\Bb) \longrightarrow \cO(\Gb_a^3)$ is given by 
\begin{align*}
x^p&\mapsto 0,& y^p &\mapsto \frac{1}{2} e^p,& u^p &\mapsto 0,&
v^p&\mapsto f^p,& \gen^p &\mapsto 1,& \zeta^{(p)} &\mapsto h^p - h.
\end{align*}

It follows $\phi$ is a injective since it maps a basis into a linearly independent set and $\psi$ is surjective since the PBW-basis of
$\cO(\Gb_a^3)$ is in its image. Clearly $\ker \psi \supseteq \cO(\Bb) \phi(\cO(\Gb))^+$, the other inclusion follows easily using the basis for $\cO(\Bb)$ given by
\begin{align*}
&\left\{x^{pn}\,y^{pr}\,(\gen^p-1)^m\,(\zeta^{(p)})^k u^{pi} v^{pj}: \, i,j,k,n,r,m\in\N_0\right\} \cup \\
&\left\{x^{pn}\,y^{pr}\,(\gen^{-p}-1)^m\,(\zeta^{(p)})^k u^{pi} v^{pj}: \, i,j,k,n,r\in\N_0, \ m\in\N \right\}.
\end{align*}
The adjoint action of $\cO(\Bb)$ is trivial, hence $\phi(\cO(\Gb))$ is invariant. The PBW-basis for $\cO(\Bb)$ and $\cO(\Gb)$ implies
$\phi$ is faithfully flat. Then Remark \ref{remark-exact-sequence-hopf} applies and the row is exact.

The middle column is exact by Proposition \ref{prop-D-tilde-extension}. Only the leftmost column remains. Let $\pi \colon \cO(\Gb) \longrightarrow \nucleo$
be as in Remark \ref{rem:algebraic-group-normal-hopfsubalgebra}, then $\ker \pi = \left\langle x^p, u^p, \gen^p - 1\right\rangle = \cO(\Gb)  \operatorname{Fr}(\cO(\Gb))^+$. 
By Remark \ref{remark-exact-sequence-hopf}, the result follows.
 \end{proof}

 \subsection{Ring-theoretical properties of $\widetilde{D}$}
 
 \begin{prop}\label{prop:ringtheoretical}
 \begin{enumerate}[leftmargin=*,label=\rm{(\roman*)}] 
 \item\label{item:ringtheoretical1} The algebra $\widetilde{D}$ admits an exhaustive ascending filtration $(\widetilde{D}_n)_{n \in \N_0}$ such that 
 $\gr \widetilde{D} \simeq \ku[T^{\pm}]\ot \ku[X_1,\dots,X_5]$. 
 
 \smallbreak
 \item\label{item:ringtheoretical2} $\widetilde{D}$ is a noetherian domain.

\smallbreak
 \item\label{item:ringtheoretical3} $\widetilde{D}$ is a PI-algebra.
 \end{enumerate} 
 \end{prop}

\begin{proof} \ref{item:ringtheoretical1}
Let $T $ be the algebra generated by $\gen^{\pm1}, \zeta, x, y, u, v$ with relations $\gen^{\pm1} \cdot \gen^{\mp 1} = 1$. 
Consider the grading on $T$ determined by
\begin{align*}
 \deg \gen^{-1} &= -1, &
\deg \gen &= \deg \zeta = 1, & \deg u &= \deg x = 2, & \deg v &= \deg y = 3.
\end{align*}
The filtration associated to this grading induces a filtration on $\widetilde{D}$
via the (evident) epimorphism $T\rightarrowdbl \widetilde{D}$. 
The relations of $\widetilde{D}$ imply that the classes of the generators commute in $\gr \widetilde{D}$ 
and $\gen \gen^{-1} = \gen^{-1} \gen = 1$. We may conclude that $\ku[T^{\pm}]\ot \ku[X_1,\dots,X_5]\rightarrowdbl \gr \widetilde{D}$. 
By dimension counting in each degree using the PBW-basis, this map is an isomorphism. 

\smallbreak
\ref{item:ringtheoretical2} follows from \ref{item:ringtheoretical1} since $\ku[T^{\pm}]\ot \ku[X_1,\dots,X_5]$ is a noetherian domain.
 
 \smallbreak
 \ref{item:ringtheoretical3} follows from Proposition \ref{prop-D-tilde-extension} and \cite[Corollary 1.13]{McRob}.
 \end{proof}

The algebra $\widetilde{H}$ is also PI by a similar reason. We observe that 
in characteristic 0, 
all finite-dimensional simple $\widetilde{H}$-modules have dimension 1 and are classified in \cite[\S 3]{abff} using results from \cite{iyudu}.

\section{Pre-Nichols algebras of the restricted Jordan plane}
\subsection{Finite-dimensional pre-Nichols algebras}
Recall that a pre-Nichols algebra of a braided vector space $(\cV,c)$ (or of its Nichols algebra)
is a graded connected braided Hopf algebra $\toba = \oplus_{n\in\N_0} \toba^n$ such that $\toba^1 \simeq \cV$ (as braided vector spaces)
and this generates $\toba$ as an algebra.
Accordingly a morphism of pre-Nichols algebras is one of a graded connected braided Hopf algebras inducing the identity on $\cV$.
Thus we have by definition  morphisms of pre-Nichols algebras $T(\cV)\twoheadrightarrow \toba \twoheadrightarrow \toba(\cV)$. 
If $\cV \in \ydk$ for some Hopf algebra $K$, then $\toba$ is a pre-Nichols algebra over $K$ if in addition $\toba \in \ydk$, that is
the kernel of  $T(\cV)\twoheadrightarrow \toba$ is a Yetter-Drinfeld submodule of $T(\cV)$

For classification issues it is important to determine all finite-dimensional pre-Nichols algebras of the restricted Jordan plane $\toba(V)$.
In this section we  classify those that factorize through the Jordan plane $\wtoba$.

\subsubsection{Pre-Nichols algebras} We shall need the
following lemma.

\begin{lemma}\label{lemma-action-and-comultiplication-Jordan}
The following formulas are valid in $\wtoba$ for every $n\in\N$:
\begin{align}
\label{g-action-y-n} g\rightharpoonup y^n &= y^n + n x y^{n-1} +\frac{1}{4} (n^2-n) x^2 y^{n-2},\\
\label{comultiplication-x-n}\Delta(x^n) &= x^n\ot 1 + 1 \ot x^n + \sum_{k=1}^{n-1} \binom{n}{k} x^{n-k}\ot x^k,\\
\label{comultiplication-y-n}\Delta(y^n) &= \sum_{k=0}^{n} \sum_{i=0}^{k} \binom{n}{k} \binom{k}{i}
(-1)^i\frac{[k-n]^{[i]}}{2^i} x^i y^{k-i} \ot y^{n-k}.
\end{align}
If $n=p\ell$ with $\ell\geq 1$, the last formula simplifies to
\begin{align}
\label{comultiplication-y-pl}\Delta(y^{p\ell}) = y^{p\ell} \ot 1 + 1 \ot y^{p\ell} + \sum_{t=1}^{\ell-1} \binom{\ell}{t} y^{p t}\ot y^{p(\ell-t)}.
\end{align}
\end{lemma}
\begin{proof}
\eqref{g-action-y-n} is proved by induction on $n$ using \eqref{jordan-relations-between-monomials}. 
Since $x$ is primitive and $c(x\ot x) = x\ot x$ \eqref{comultiplication-x-n} follows.
\eqref{comultiplication-y-n} is proved  by induction using \eqref{g-action-y-n} and that $g^k\rightharpoonup y = y+kx$ for every $k\in\N$.
To prove \eqref{comultiplication-y-pl}, notice that $y^p$ is primitive because $\binom{p}{k} = 0$ for $1\leq k \leq p-1$ and \eqref{comultiplication-y-n}. By \eqref{g-action-y-n} $c(y^p\ot y^p) = y^p \ot y^p$, so \eqref{comultiplication-y-pl} follows.
\end{proof}

\begin{prop}
Let $k,\ell\in\N$ and $a\in\ku$. The algebras
\begin{align*}
\mathcal{K}(k,a) &\coloneqq \wtoba/(y^{p^k}-a x^{p^k}), \\
\mathcal{F}(\ell) &\coloneqq \wtoba/(x^{p^\ell}), \\
\mathcal{G}(k,\ell,a)&\coloneqq \wtoba/(y^{p^k}-a x^{p^k}, x^{p^\ell})
\end{align*}
are pre-Nichols algebras of $\toba(V)$ over $\ku\Gamma$ that factorize by $\wtoba$ and
have PBW-bases given by the following table.
\begin{align*}
\begin{tabular}{c|c}
Algebra & PBW-basis\\
\hline
$\mathcal{K}(k,a)$ & $\{x^i y^j \colon i\in\N_0, j\in\I_{0, p^k-1}\}$ \\
\hline
$\mathcal{F}(\ell)$ & $\{x^i y^j \colon i\in\I_{0, p^\ell-1}, j\in\N_0\}$ \\
\hline
$\mathcal{G}(k,\ell,a)$ & $\{x^i y^j \colon i\in\I_{0, p^\ell-1}, j\in\I_{0, p^k-1}\}$
\end{tabular}
\end{align*}
\end{prop}

\begin{proof}
The elements of the form $f=b_1 x^{p^t} + b_2 y^{p^t}$ with $b_1,b_2\in\ku$ and $t\in\N$ are
primitive, $g\rightharpoonup f = f$ and $c(f\ot f)= f\ot f$. Thus the algebras $\mathcal{K}(k,a)$, $\mathcal{F}(\ell)$ and $\mathcal{G}(k,\ell,a)$ are indeed pre-Nichols algebras. We prove the existence
of the PBW-bases using the diamond lemma. We declare $x<y$ and we define a ordering in the monomials with letter $x$, $y$ in the following way:
\begin{itemize}[leftmargin=*]
\item $X<Y$ if the length of $X$ (the number of letters in the product $X$) is less than the length
of $Y$.
\item If $X$ and $Y$ have the same length, but the number of $y$'s in $Y$ is greater than the
number of $y$'s in $X$, then $X<Y$.
\item If $X$ is a permutation of the letters of $Y$, but has a lower number of inverses, then
$X<Y$.
\end{itemize}
Here, we say that the monomial $X=x_1\cdots x_s$ with $x_t\in\{x,y\}$, $t\in\I_{1, s}$ has inverse
$(i,j)$, $1\leq i<j\leq s$ if $x_i>x_j$. With this order, if all ambiguities can be solved, then the hypotheses 
of the diamond lemma are fulfilled. Hence the sets of irreducible monomials are basis of each algebra respectively,
 that is the proposed PBW-bases. Using \eqref{jordan-relations-between-monomials} the ambiguity resolutions are:
\begin{itemize}[leftmargin=*]
\item $y^{p^k-1} y y^{p^k-1}$ for $\mathcal{K}(k,a)$ and $\mathcal{G}(k,\ell,a)$:
\begin{align*}
(y^{p^k-1} y) y^{p^k-1} &= a x^{p^k} y^{p^k-1} = a y^{p^k-1} x^{p^k} = y^{p^k-1} (y y^{p^k-1}).
\end{align*}
\item $y^{p^k-1} y x$ for $\mathcal{K}(k,a)$ and $\mathcal{G}(k,\ell,a)$:
\begin{align*}
(y^{p^k-1} y) x &= y^{p^k} x = a x^{p^k+1}.\\
y^{p^k-1}(y x) &= y^{p^k-1}xy - \frac{1}{2} y^{p^k-1}x^2 =\\ &=\sum_{t=0}^{p^k-1}\binom{p^k-1}{t}\frac{(-1)^t}{2^t}x^{1+t} y^{p^k-t-1}([1]^{[t]}y-\frac{1}{2}[2]^{[t]} x)\\
&= x y^{p^k} - \frac{(p^k)!}{2^{p^k}} x^{p^k+1} + \sum_{t=1}^{p^k-1}\binom{p^k}{t}(-1)^t \frac{[1]^{[t]}}{2^t} x^{1+t} y^{p^k-t} \\
&= x y^{p^k} = a x^{p^k+1}.
\end{align*}
\item $x^{p^\ell-1} x x^{p^\ell-1}$ for $\mathcal{F}(\ell)$ and $\mathcal{G}(k,\ell,a)$:
\begin{align*}
(x^{p^\ell-1} x) x^{p^\ell-1} = 0 = x^{p^\ell-1} (x x^{p^\ell-1}).
\end{align*}

\item $y x x^{p^\ell-1}$ for $\mathcal{F}(\ell)$ and $\mathcal{G}(k,\ell,a)$:
\begin{align*}
(yx)x^{p^\ell -1} = xyx^{p^\ell-1} - \frac{1}{2} x^{p^\ell + 1} = x^{p^\ell} y + \frac{1}{2}x^{p^\ell} =0 = y (x x^{p^\ell-1}).
\end{align*}
\end{itemize}
\end{proof}

\subsubsection{Exhaustion} 
We show now that the  algebras $\mathcal{G}(k,\ell,a)$ are the only finite-dimensional
pre-Nichols algebras of $\toba(V)$ that factorize through $\wtoba$. The relations of a
pre-Nichols algebra with minimal degree are primitive, thus we start by computing the primitive elements of
the algebras $\wtoba$, $\mathcal{K}(k,a)$, $\mathcal{F}(\ell)$ or $\mathcal{G}(k,\ell,a)$.
Let us denote by $\cA$ one of these algebras and by $\Pc(\cA)^n$  
the space of primitives in degree $n$ which is a Yetter-Drinfeld submodule over $\Gamma$ of $\cA^n$. 
Since $\Gamma$ acts unipotently if $\Pc(\cA)^n \neq 0$,  there exists $0 \neq f\in\Pc(\cA)^n$ invariant. 
We then first compute the invariant elements of $\cA^n$.

\begin{lemma}\label{invariants-of-A}
Let $n\in\N$. 
\begin{enumerate}[leftmargin=*,label=\rm{(\roman*)}]
	\item\label{lemma:invariants-of-A-1} If $\cA$ is either $\wtoba$, $\mathcal{K}(k,a)$ or $\mathcal{G}(k,\ell,a)$ with $k\leq\ell$, then the
	submodule $\mathcal{I}^n\subseteq \cA^n$ of invariant elements of degree $n$ is generated by the set
	\begin{align*}
	\left\{x^{n-p\ell}y^{p\ell}\colon 0\leq \ell \leq \lfloor{\frac{n}{p}}\rfloor \right\}.
	\end{align*}
	\item\label{lemma:invariants-of-A-2} If $\cA$ is $\mathcal{F}(\ell)$ or $\mathcal{G}(k,\ell,a)$ with $k>\ell$, then the
	submodule $\mathcal{I}^n\subseteq \cA^n$ of invariant elements of degree $n$ is generated by the set
	\begin{align*}
	\left\{x^{p^{\ell}-1} y^{n+1-p^{\ell}}\right\}\cup\left\{x^{n-p\ell}y^{p\ell}\colon 0\leq \ell \leq \lfloor{\frac{n}{p}}\rfloor \right\}, & &\text{If }n\geq p^{\ell}-1,\\
	\left\{x^{n-p\ell}y^{p\ell}\colon 0\leq \ell \leq \lfloor{\frac{n}{p}}\rfloor \right\}, & & \text{If }n< p^{\ell}-1.
	\end{align*}
\end{enumerate}
\end{lemma}

\begin{proof}
\ref{lemma:invariants-of-A-1} We just consider the case $\cA = \wtoba$, as all other algebras have similar PBW-bases. 
By \eqref{g-action-y-n},  $x^{n-p\ell}y^{p\ell} \in \mathcal{I}^n$ for all $\ell$ as above.

\smallbreak
Let $f = \sum_{i=0}^{n} b_i x^i y^{n-i} \in \mathcal{I}^n$. By \eqref{g-action-y-n}  we get
\begin{align*}
0 &= g \rightharpoonup f - f = \\
b_0 n x y^{n-1} &+ \sum_{i=1}^{n-1}\left[b_i(n-i) + \frac{b_{i-1}}{4}((n-i+1)^2 - (n-i+1))\right] x^{i+1} y^{n-i-1}.
\end{align*}
And thus $b_0 = 0$ if $p\nmid n$ and
\begin{align}\label{eq:bi-invariant}
0 &= b_i(n-i) + \frac{b_{i-1}}{4} \left((n-i+1)^2 - (n-i+1) \right) & & \forall i\in\I_{n-1}.
\end{align}
Set $r_i := n-i$ in $\fp$ for $i\in\I_{n-1}$. If $r_i =0$, then \eqref{eq:bi-invariant} gives no restriction.
If $r_i \neq 0$, then \eqref{eq:bi-invariant} says that
$
b_i =- \frac{1}{4} b_{i-1} (r_i+1)$. 
Therefore
\begin{itemize}[leftmargin=*]
\item If $r_i = -1$, then $b_i = 0$.
\item If $r_i \neq -1,0$ and $b_{i-1} = 0$,  then $b_{i} = 0$.
\end{itemize}

Assume that $p\nmid n$. Then $b_0 =0$. Then $r_1 \neq -1$; thus either $r_1 \neq 0$  hence $b_1 =0$, or else 
$r_1 = 0$ in which case $r_2 = -1$ and $b_2 =0$. 
Arguing recursively $b_i =0$ for any $i$  until $r_i = 0$.
i.e. $n- i =p\ell$ for some $\ell$. For this $i$ we have   $r_{i+1} = -1$, hence $b_{i+1} = 0$ and so on.
That is, $f$ is sum of monomials as desired.
Finally, if $p \mid n$, $r_{1} = -1$, hence $b_{1} = 0$ and we continue analogously.

\medbreak
\ref{lemma:invariants-of-A-2} If $n< p^{\ell}-1$, the same recursive argument as in \ref{lemma:invariants-of-A-1} applies.
If $n\geq p^{\ell}-1$, by \eqref{g-action-y-n} the element $x^{p^{\ell}-1} y^{n+1-p^{\ell}}$ is invariant. Then taking 
$f\in \mathcal{I}^n$ as a linear combination of the remaining PBW-basis elements leads to a similar recursive argument as in
 \ref{lemma:invariants-of-A-1}.
\end{proof}

\begin{prop}\label{prop:prop:calculation-of-primitives-in-pre-nichols}
Let $\cA$ be any of the algebras $\wtoba$, $\mathcal{K}(k,a)$, $\mathcal{F}(\ell)$ or $\mathcal{G}(k,\ell,a)$.
Let $n\in\N$. Then
\begin{align*}
\Pc(\cA)^n = \begin{cases}
0 & \text{if } n \text{ is not a power of } p.\\
\ku x^{p^t} + \ku y^{p^t} & \text{if } n=p^t \text{ for some } t\in\N_0.
\end{cases}
\end{align*}
\end{prop}

\begin{proof}
	
We proceed in steps. First we show that if $h\in\Pc(A)^n\cap \cI^n$, then $h$ is in the linear span of $\left\{x^{n-p\ell}y^{p\ell}\colon 0\leq \ell \leq \lfloor{\frac{n}{p}}\rfloor \right\}$ regardless of $\cA$. If $\cA$ is $\mathcal{F}(\ell)$ or $\mathcal{G}(k,\ell,a)$ with $k>\ell$, and
$n\geq p^{\ell}-1$, then $h$ is of the form
\begin{align*}
h = bx^{p^{\ell}-1} y^{n+1-p^{\ell}} + \sum_{\ell=0}^{\lfloor{\frac{n}{p}}\rfloor} b_\ell x^{n-p\ell} y^{p\ell},\qquad b,b_\ell\in\ku, \ell\in\I_{0, \lfloor{\frac{n}{p}}\rfloor}.
\end{align*}
Since $\Delta(h) - h\ot 1 - 1\ot h = 0$, then $b x^{p^{\ell}-1}\ot y^{n+1-p^{\ell}} + \Theta = 0$, where $\Theta$ is a linear combination
of tensors linearly independent to $x^{p^{\ell}-1}\ot y^{n+1-p^{\ell}}$. Then $b = 0$ and $h$ is as desired.

Now, we show that if we have an non zero element $r = a x^n + b y^n$, $a$, $b\in\ku$, then $r$ is primitive if and only if $n=p^t$ for some $t\in\N_0$. By Lucas theorem
$\binom{n}{k} = 0$ for all $k \in \I_{n-1}$ if and only if $n=p^t$ for some $t\in\N_0$. Thus the claim
for $x^n$ is valid and we can assume $b\neq 0$. In this case $r$ is never primitive if $p\nmid n$ and $n\neq 1$, because $\Delta(y^n)$ has as a summand $n y\ot y^{n-1}$. If $n=p\ell$, $r$ is primitive if and only if $\binom{\ell}{k} = 0$ for all $k \in \I_{\ell -1}$, and that happens if and only if $\ell$ is a power of $p$.

Now we classify all the primitives. As in the previous proof we just deal with $\cA = \wtoba$, the other algebras have similar arguments. 
Since the action of $g$ on $\Pc(\cA)^n$ is unipotent,
 if $\Pc(\cA)^n \neq 0$, then there exists $0\neq f\in\Pc(A)^n$ such that $g\rightharpoonup f = f$. 
By Lemma \ref{invariants-of-A} $f$ is of the form
\begin{align*}
f &= \sum_{\ell=0}^{\lfloor{\frac{n}{p}}\rfloor } c_\ell x^{n-p\ell} y^{p\ell} \qquad c_\ell \in\ku\, \forall \ell. 
\end{align*}

Since $\Delta(f) - f\ot 1 - 1\ot f = 0$, there exists  is a linear combination $\Xi_1$ of tensors linearly independent to $x^{n-p\ell}\ot y^{p\ell}$ and $y^{p\ell} \ot x^{n-p\ell}$ such that
\begin{align*}
0 = \Xi_1+\sum_{\substack{1\leq \ell \leq \lfloor{\frac{n}{p}\rfloor}\\ p\ell \neq n}} c_\ell (x^{n-p\ell}\ot y^{p\ell} + y^{p\ell} \ot x^{n-p\ell}).
\end{align*}
 Then $c_\ell = 0 $ for $1\leq \ell \leq \lfloor\frac{n}{p}\rfloor$ and $ p\ell \neq n$. If $p\nmid n$, $f = c_0 x^{n}$, but this is primitive only if $n$ is a power of $p$, so $\Pc(\cA)^n = 0$. If $n=pj$ for $j\in\N$, then $f = c_0 x^{pj} + c_{j} y^{pj}$. This is primitive if and only if $j$ is a power of $p$.
We then have $\Pc(\cA)^n = 0$ if $n$ is not a power of $p$. If $n=p^t$ with $t\in\N_0$, then $\ku x^{n} + \ku y^{n}$ are exactly the invariant primitives. 
We want to see that these are the only primitives in $\cA$. Fix $n =p^t$ for
$t\in\N_0$. If $\ku x^{p^t} + \ku y^{p^t} \subsetneq \Pc(\cA)^{p^t}$, the canonical Jordan form of $g$ in 
$\Pc(\cA)^{p^t}$ implies the existence of a non-invariant primitive $f$ such that $g\rightharpoonup f 
\overset{(\star)}{=} f + h$ with $h =a_0 x^{p^t} + a_1 y^{p^t} \neq 0$ an invariant primitive. We can assume then $f$ is a solution of the equation 
$(\star)$.  If $f = \sum_{i=0}^{p^t} \widetilde{b}_i x^i y^{n-i}$ , then we have by \eqref{g-action-y-n}
\begin{align*}
\sum_{i=1}^{p^t-1}\left[- i \widetilde{b}_i  + \frac{\widetilde{b}_{i-1}}{4}((-i+1)^2 - (-i+1))\right] x^{i+1} y^{p^t-i-1} & = a_0 x^{p^t} + a_1 y^{p^t}.
\end{align*}
 The same argument as at the end of the proof of Lemma \ref{invariants-of-A} shows that $\widetilde{b}_i=0$ for $p\nmid i$ and $i\neq p^t - 1$. 
 That is, $f$ has the form
\begin{align*}
f = b x^{p^t-1}y + \sum_{\ell=0}^{p^{t-1}} b_\ell x^{n-p\ell} y^{p\ell}, \qquad b,b_\ell\in\ku, \ell\in\I_{0, p^{t-1}}.
\end{align*}
Using now that $f$ is primitive,  we get
$b x^{p^t-1}\ot y + \Xi_2 = 0$, 
where $\Xi_2$ is a linear combination of tensors linearly independent of $x^{p^t-1}\ot y$. Then $b=0$ and 
$\Pc(\cA)^{p^t}= \ku x^{p^t} + \ku y^{p^t}$.

\end{proof}

\begin{lemma}\label{lema:G-lk-iso}
 $\mathcal{G}(k,\ell,a)\simeq \mathcal{G}(k,\ell,0)$ as braided Hopf algebras.
\end{lemma}
\begin{proof}
Let $t\in \ku$. The automorphism of Yetter-Drinfeld modules $\psi_t: V \to V$ given by
$x \mapsto x$, $y \mapsto y+tx$
induces an automorphism of braided Hopf algebras $\Psi_t: T(V) \to T(V)$ that descends to $\Psi_t:\wtoba \to \wtoba$ because 
$\Psi(yx-xy+\frac{1}{2} x^2) = yx-xy+\frac{1}{2} x^2$. Thus we have a  morphism of groups $\Gb_a \to \Aut \wtoba$, $t \mapsto \Psi_t$.

Let now $t\in\ku$ be a solution of the equation $(t^p-t)^{p^{k-1}} + a = 0$. Arguing recursively
we prove  that the following equalities hold in $\wtoba$:
\begin{align*}
(y+tx)^n &= y^n + t^n x^n + \sum_{i=1}^{n-1} \sum_{j=i}^{n}\binom{n}{j}\stirling{j}{i} (-2)^{i-j}
t^i  x^j y^{n-j},&  n&\in\N.
\end{align*}
Hence $(y+tx)^p = y^p + (t^p-t) x^p$ by \eqref{eq:stirling}. 
Since $x^p$ commutes with $y^p$, we get $(y+tx)^{p^{k}} = y^{p^{k}} + (t^p-t)^{p^{k-1}} x^{p^k} = y^{p^{k}} - a x^{p^k}$ . 
Then $\Psi$ induces a morphism of braided Hopf algebras $\mathcal{G}(k,\ell,0) \longrightarrow \mathcal{G}(k,\ell,a)$. 
Repeating the argument with $\Psi^{-1}$ we conclude that $\mathcal{G}(k,\ell,a) \simeq \mathcal{G}(k,\ell,0)$.  
\end{proof}

Because of the previous Lemma, we introduce $\mathcal{G}(k,\ell)  \coloneqq \mathcal{G}(k,\ell, 0)$. 
We now state the main result of this Section.

\begin{theorem}\label{thm;prenichols}
If $\toba$ is a finite-dimensional pre-Nichols algebra of $V$ that factorizes through $\wtoba$, 
then $\toba \simeq \mathcal{G}(k,\ell)$ for unique $k,\ell\in\N$.
\end{theorem}

\begin{proof}
By assumption, there is a morphism $\pi: \wtoba\twoheadrightarrow \toba$ of pre-Nichols algebras. Since $\dim \wtoba = \infty$, $\ker \pi \neq 0$. 
Pick $0 \neq f\in \ker \pi$ homogeneous of minimal degree $m$;  then $f\in \Pc(\wtoba)$. 
By Proposition \ref{prop:prop:calculation-of-primitives-in-pre-nichols} , $m = p^k$,  $k\in\N$, and
there exists  $(b_1,b_2) \in\ku^2 - 0$ such that $f = b_1 x^{p^k} + b_2 y^{p^k}$. We have now two cases, $b_2 \neq 0$ and $b_2 = 0$.

\smallbreak
If $b_2\neq 0$, taking $a = -\frac{b_1}{b_2}$ we get a morphism of pre-Nichols algebras $\pi_1: \mathcal{K}(k,a)\twoheadrightarrow \toba$.
Since $\dim \mathcal{K}(k,a) = \infty$, $\ker \pi_1 \neq 0$. 
Pick $0 \neq f_1\in \ker \pi_1$ homogeneous of minimal degree $m_1$;  then $f_1\in \Pc(\mathcal{K}(k,a))$. 
By Proposition \ref{prop:prop:calculation-of-primitives-in-pre-nichols} , $m_1 = p^{\ell}$,  $\ell\in\N$, and
there exists  $(c_1,c_2) \in\ku^2 - 0$ such that $f_1 = c_1 x^{p^\ell} + c_2 y^{p^\ell} \in \ker \pi_1$. 
Now the preimage of $f_1$ in $\wtoba$ belongs to $\ker \pi$, hence  $\ell\geq k$ by the minimality of $m = p^k$. 
So $f_1 = c_1 x^{p^\ell} + c_2 y^{p^\ell} = (c_1 + c_2 a^{p^{\ell-k}})x^{p^\ell}$. Hence $0 \neq x^{p^\ell} \in \ker \pi_1$ and 
we get a morphism of pre-Nichols algebras
$\pi_2: \mathcal{G}(k,\ell,a)\twoheadrightarrow \toba$ which is actually an isomorphism. Otherwise
pick $0 \neq f_2\in \ker \pi_2$ homogeneous of minimal degree $m_2 = p^h$;  then $f_1\in \Pc(\mathcal{G}(k, \ell,a))$. 
Now the preimage of $f_2$ in $\mathcal{K}(k,a)$ belongs to $\ker \pi_1$, hence $h \geq \ell\geq k$, 
but all the primitives in $\mathcal{G}(k,\ell,a)$ have degree $< \max\{p^\ell,p^k\}$. So $\toba\simeq\mathcal{G}(k,\ell,a) \simeq \mathcal{G}(k,\ell)$
by Lemma \ref{lema:G-lk-iso}.

\smallbreak
If $b_2= 0$, then we get a morphism of pre-Nichols algebras $\pi_3: \mathcal{F}(k)\twoheadrightarrow \toba$. 
Since $\dim \mathcal{F}(k) = \infty$, $\ker \pi_3 \neq 0$. 
Pick $0 \neq f_3\in \ker \pi_3$ homogeneous of minimal degree $m_3$;  then $f_3\in \Pc(\mathcal{F}(k))$, 
$m_3 = p^{\ell}$,  $\ell\in\N$, and
there exists  $(c_1,c_2) \in\ku^2 - 0$ such that $f_3 = c_1 x^{p^\ell} + c_2 y^{p^\ell} \in \ker \pi_3$. 
Since the preimage of $f_3$ in $\wtoba$ belongs to $\ker \pi$,  $\ell\geq k$. 
So $f_3 =  c_2 y^{p^\ell} $. Hence $0 \neq y^{p^\ell} \in \ker \pi_3$ and 
we get an isomorphism 
$\pi_4: \mathcal{G}(\ell, k, 0)\twoheadrightarrow \toba$  arguing as above.
\end{proof}

\subsubsection{The poset} We have the following picture in the poset of pre-Nichols algebras of the restricted Jordan plane:

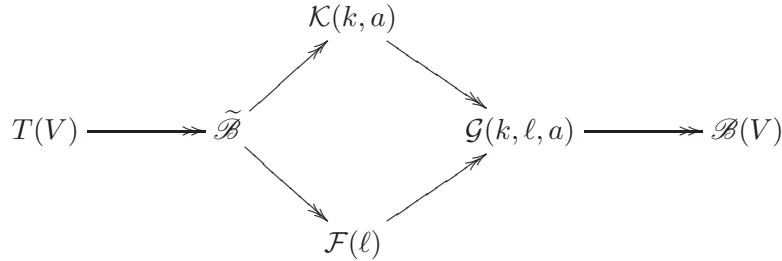
\begin{align*}
\xymatrix@C-5pt{ 
 && &\mathcal{K}(k,a) \ar @{->>}[1,1] & & &
\\
T(V) \ar @{->>}[0,2]   & & \wtoba \ar@{->>}[-1,1] \ar@{->>}[1,1] & &  \mathcal{G}(k,\ell,a) \ar@{->>}[0,2]  & & \toba(V)
\\ & & & \mathcal{F}(\ell) \ar @{->>}[-1,1] & &
}
\end{align*}

For two finite dimensional pre-Nichols algebras $\cR_1$ and $\cR_2$ from the previous families, we say $\cR_1 \geq \cR_2$ if and only if there exist a pre-Nichols algebra
epimorphism $\cR_1\twoheadrightarrow \cR_2$. This is a well defined poset since if $\cR_1 \geq \cR_2$ and $\cR_1 \leq \cR_2$,
then $\cR_1 = \cR_2$ by the respective PBW bases and because a pre-Nichols algebra morphism is the identity in $V$.

Although $\mathcal{G}(k,\ell,a) \simeq \mathcal{G}(k,\ell)$ as braided Hopf algebras, they are different as pre-Nichols algebras
if $a\neq 0$ and $k< \ell$. This is because we require pre-Nichols algebras morphisms to be the identity in $V$. We then reserve the 
notation $\mathcal{G}(k,\ell,a)$ only for the case $a\neq 0$ and $k < \ell$.

The aim of this subsection is the complete description of this poset.

\begin{lemma}
The following comparisons in the poset are valid:
\begin{enumerate}[leftmargin=*,label=\rm{(\roman*)}] 
	\item\label{lemma:pre-Nichols-comparison-1} $\mathcal{G}(k,\ell,a) \geq \mathcal{G}(k',\ell',b)$ iff one of the
	following conditions holds:
	\begin{enumerate}
		\item $\ell \geq \ell'$, $\ell' > k \geq k'$ and $a = b^{p^{k-k'}}$.
		\item $\ell \geq \ell'$ and $k\geq\ell'>k'$.
	\end{enumerate}
	\item\label{lemma:pre-Nichols-comparison-2} $\mathcal{G}(k,\ell,a) \geq \mathcal{G}(k',\ell')$ if and only if $\ell \geq \ell'$,
	$k\geq \ell'$ and $k \geq k'$.
	\item\label{lemma:pre-Nichols-comparison-3} $\mathcal{G}(k,\ell) \geq \mathcal{G}(k',\ell',a)$ if and only if $\ell \geq \ell'$
	and $k\geq \ell'>k'$.
	\item\label{lemma:pre-Nichols-comparison-4} $\mathcal{G}(k,\ell) \geq \mathcal{G}(k',\ell')$ if and only if $\ell \geq \ell'$ and
	$k \geq k'$.
\end{enumerate}
\end{lemma}

\begin{proof}
	For simplicity we denote by $\sx$, $\sy$ the corresponding generators in $\mathcal{G}(k',\ell',a)$ or $\mathcal{G}(k',\ell')$, and
	by $x$, $y$ the ones in $\mathcal{G}(k,\ell,a)$ or $\mathcal{G}(k,\ell)$.
	
	\ref{lemma:pre-Nichols-comparison-1} If any of the two conditions for the constants $\ell,\ell',k,k',a,b$ holds, then clearly the
	corresponding pre-Nichols algebra morphism exists. 
	Let $\pi\colon\mathcal{G}(k,\ell,a) \twoheadrightarrow \mathcal{G}(k',\ell',b)$ be a
	pre-Nichols algebras morphism. Then $\pi(x^{p^\ell}) = \sx^{p^{\ell}} = 0$, hence $\ell\geq\ell'$. If $k\geq \ell' > k'$ we are done,
	if $\ell' > k \geq k'$, then $\sy^{p^{k'}} - b \sx^{p^{k'}} = 0$ hence $\sy^{p^{k}} - b^{p^{k-k'}} \sx^{p^{k}} = 0$. So $0 = \pi(y^{p^{k}} - a x^{p^{k}}) = \sy^{p^{k}} - a \sx^{p^{k}} = (b^{p^{k-k'}}-a)\sx^{p^{k}}$, this implies $b^{p^{k-k'}}=a$. 
	Finally if $k<k'$, then $\pi(y^{p^{k}} - a x^{p^{k}}) = \sy^{p^{k}} - a \sx^{p^{k}} = 0$, and this implies $a=0$ a contradiction.
	
	\ref{lemma:pre-Nichols-comparison-2} If $\ell \geq \ell'$,$k\geq \ell'$ and $k \geq k'$, then clearly the
	corresponding morphism exists. Let 
	$\pi\colon\mathcal{G}(k,\ell,a) \twoheadrightarrow \mathcal{G}(k',\ell')$ be a pre-Nichols algebras morphism.  The same argument as
	in \ref{lemma:pre-Nichols-comparison-1} shows  $\ell\geq\ell'$. If $k<\ell'$ then $\pi(y^{p^k}-a x^{p^k}) = 0 = \sy^{p^k} - a \sx^{p^k}$, and this implies $a=0$ since $\sx^{p^k} \neq 0$, a contradiction. Then $k\geq \ell'$ and $\pi(y^{p^k}-a x^{p^k}) = \sy^{p^k} = 0$ hence
	$k\geq k'$.
	
	\ref{lemma:pre-Nichols-comparison-3} If $\ell \geq \ell'$ and $k\geq \ell'>k'$ then the corresponding morphism exist. Now
	let $\pi\colon\mathcal{G}(k,\ell) \twoheadrightarrow \mathcal{G}(k',\ell',a).$ By the same argument as in \ref{lemma:pre-Nichols-comparison-1}, $\ell \geq \ell'$. If $k\geq \ell'>k'$ we are done. If $\ell' > k \geq k'$, then
	$\pi(y^{p^k}) = 0 = \sy^{p^k}$. Since $0 = \sy^{p^{k'}} - a \sx^{p^{k'}}$ implies 
	$0 = \sy^{p^k} - a^{p^{k-k'}}\sx^{p^k} = a^{p^{k-k'}}\sx^{p^k}$.
	Then $a=0$, a contradiction. Only the case $k'> k$ remains. In this case $\pi(y^{p^k}) = 0 = \sy^{p^k}$, but $\sy^{p^k}\neq 0$, a
	a contradiction.
	
	\ref{lemma:pre-Nichols-comparison-4} If $\ell \geq \ell'$ and $k\geq k'$ then the corresponding morphism exist. Now
	let $\pi\colon\mathcal{G}(k,\ell) \twoheadrightarrow \mathcal{G}(k',\ell').$ By the same argument as in \ref{lemma:pre-Nichols-comparison-1}, $\ell \geq \ell'$. Now since $\pi(y^{p^k}) = 0 = \sy^{p^k}$ then $k\geq k'$.
\end{proof}

\subsection{Finite-dimensional Hopf algebras}

By bosonization we obtain new examples of Hopf algebras.

\begin{coro}\label{coro:pre-Nichols-bosonization}
\begin{enumerate}[leftmargin=*,label=\rm{(\roman*)}] 
\item The Hopf algebras $\mathcal{K}(k,a)\# \ku \Gamma$ and $\mathcal{F}(\ell)\# \ku \Gamma$   have \newline $\GK = 1$.  

\item The Hopf algebras $H_{k,\ell,a} = \mathcal{G}(k,\ell,a) \# \ku \Gamma$   have dimension $p^{\ell + k + 1}$. \qed 
\end{enumerate} 

\end{coro}

\section{The graded dual of the Jordan plane}
In this Section $\car \ku \neq 2$.
We present the graded dual $\cE$ of the Jordan plane $\widetilde{\toba}$ by generators and relations. In this section the duals are
calculated in the category $\ydG$, hence the canonical identification $(U_1\ot U_2)^* \simeq U_2^* \ot U_1^*$ for $U_1$, $U_2\in\ydG$ is used.
Let $(\alpha_{i,j})_{i,j\in\N_0}$ be the dual basis of the basis $(x^i y^j)_{i,j\in\N_0}$ of $\widetilde{\toba}$. 
Clearly $\cE$ is linearly spanned by the $\alpha_{i,j}$'s.

\begin{prop}\label{prop:graded-dual}
The braided Hopf algebra $\cE$ is presented by generators $\x^{[n]}$, $\yt^{[n]}$, $n\in\N$, and relations
\begin{align}\label{eq:rels-graded-dual}
\begin{aligned}
\x^{[n]} \x^{[m]} &= \binom{n+m}{n} \x^{[n+m]}, \qquad \yt^{[n]} \yt^{[m]} = \binom{n+m}{n} \yt^{[n+m]},\\
\x^{[n]} \yt^{[m]} &= \sum_{k=0}^{m} \binom{n+k}{k} (-1)^k \frac{[-n]^{[k]}}{2^k} \yt^{[m-k]} \x^{[n+k]}, \qquad \x^{[0]} = \yt^{[0]} = 1,
\end{aligned}
\end{align}
for all $n,m\in \N_0$.
The family $(\yt^{[m]} \x^{[n]})_{n,m\in\N_0}$ is a basis of $\cE$. The coproduct and the braiding are given by
\begin{gather}\label{eq:coprod-graded-dual}
\begin{aligned}
\Delta(\x^{[n]}) &= \sum_{k=0}^{n} \x^{[k]} \ot \x^{[n-k]}, \\
\Delta(\yt^{[n]}) &= \sum_{k=0}^{n}\sum_{i=0}^{k}(-1)^i\frac{[n-k]^{[i]}}{2^i} \yt^{[n-k]}\ot \yt^{[k-i]} \x^{[i]} .
\end{aligned}
\\
\label{eq:braiding-graded-dual}
\begin{split}
c(\yt^{[j]} \x^{[i]}  \ot \yt^{[m]} \x^{[n]})& = \\
\yt^{[m]} \x^{[n]}  &\ot 
\sum_{k=0}^{j} \binom{k+i}{k}(-1)^k\frac{[2(n+m)]^{[k]}}{2^k}  \yt^{[j-k]} \x^{[k+i]}, 
\end{split} 
\end{gather}
for all $n,m,i,j\in\N_0$.
\end{prop}
\begin{proof}
Let $A$ be the algebra presented as above. It is graded with $\deg \x^{[n]} = \deg \yt^{[n]} = n$ for every $n$.
The elements 
\begin{align*}
\x^{[n]} &= \alpha_{0,n}& &\text{ and }& \yt^{[n]} &= \alpha_{n,0},& n&\in\N,
\end{align*}
of $\cE$ satisfy the relations \eqref{eq:rels-graded-dual}, hence we have a graded epimorphism $A\twoheadrightarrow \cE$. 
By dimension counting in each degree this is an isomorphism.
A direct calculation shows that $\yt^{[m]} \x^{[n]}  = \alpha_{m,n}$, hence $(\yt^{[m]} \x^{[n]})_{n,m\in\N_0}$ is a basis of $\cE$. 
The coproduct and braiding formulas follow  in a straightforward way.
\end{proof}

In the previous proof, the following coproduct formula in $\widetilde{\toba}$ is useful for computations.
Given  $n\in\N_0$, $\ell\in\I_{0, n}$, we have
\begin{align*}
&\Delta(x^{n-\ell} y^\ell) = \\
&\sum_{k=0}^{n-\ell}\sum_{t=0}^\ell\sum_{i=0}^t \binom{n-\ell}{k}
\binom{\ell}{t} \binom{t}{i}\frac{(-1)^i}{2^i} [t-\ell-2k]^{[i]} x^{n+i-\ell-k} y^{t-i}\ot x^k y^{\ell-t}.
\end{align*}

\bigbreak
Let $\mathfrak{G}(k,\ell) = \mathcal{G}(k,\ell)^*$; this is a post-Nichols algebra of $(W,c)$.

\begin{coro}\label{coro:divided-powers}
\begin{enumerate}[leftmargin=*,label=\rm{(\roman*)}] 
\item\label{item:divided-power1} If $\car \ku =0$, then $\cE\simeq \ttoba$ as braided Hopf algebras.

\item\label{item:divided-power2} If $\car \ku = p > 2$, then 	$\cE = \bigcup _{k, \ell \in \N}\mathfrak{G}(k,\ell)$.

\end{enumerate} 
\end{coro}
\begin{proof}
	\ref{item:divided-power1} In $\car \ku = 0$, $\x^{[n]} = \frac{1}{n!}(\x^{[1]})^n$ and $\yt^{[n]} = \frac{1}{n!} (\yt^{[1]})^n$ for every
	$n\in\N_0$. Hence $\cE$ is presented by the  elements $\x^{[1]}, \yt^{[1]}$ with relation 
	\begin{align*}
	\yt^{[1]}\x^{[1]} - \x^{[1]} \yt^{[1]} + \frac{1}{2} (\x^{[1]})^2 = 0.
	\end{align*}
	Then $\cE\simeq \ttoba$ via the isomorphism of algebras given by $\x^{[1]}\mapsto u$ and $\yt^{[1]}\mapsto -v$. This
	is actually an isomorphism of braided Hopf algebras since $\x^{[1]}$ and $\yt^{[1]}$ are primitive in $\cE$, and the
	braiding $c$ between them is given by
	\begin{align*}
	c(\x^{[1]}\ot \x^{[1]}) &= \x^{[1]}\ot \x^{[1]}, & c(\x^{[1]}\ot \yt^{[1]}) &= \yt^{[1]}\ot \x^{[1]}, \\
	c(\yt^{[1]}\ot \x^{[1]}) &= \x^{[1]} \ot (-\x^{[1]}+\yt^{[1]}), & c(\yt^{[1]}\ot \yt^{[1]}) &= \yt^{[1]}\ot (\yt^{[1]} - \x^{[1]}).
	\end{align*}
	Hence the claim follows.
	
	\ref{item:divided-power2} Since $\widetilde{\toba}\twoheadrightarrow  \mathcal{G}(k,\ell)$, then 
	$\mathfrak{G}(k,\ell) \hookrightarrow \cE$. Hence  $\bigcup \mathfrak{G}(k,\ell) \subseteq \cE$.
	Now since $\cE = \oplus_{n\in\N_0} \cE^n$, it is enough to show that for every $N\in\N_0$ there exist $k,\ell\in\N$ such that
	$\cE^N\subseteq \mathfrak{G}(k,\ell)$. Fix $N\in\N_0$ and take $k,\ell\in\N$ such that $\min\{p^k,p^\ell\} > N$. Then using the
	PBW basis of $\mathcal{G}(k,\ell)$ and $\widetilde{\toba}$,  $\mathcal{G}(k,\ell)^N$ is isomorphic as a vector space to $\widetilde{\toba}^N$ hence $\cE^N \simeq \mathfrak{G}(k,\ell)^N \subseteq \mathfrak{G}(k,\ell)$.
\end{proof}

\begin{coro}\label{coro:presentation-post-nichols}
Let $k, \ell \in \N$.
The braided Hopf algebra $\mathfrak{G}(k,\ell)$ is presented by generators $\x^{[n]}$, $\yt^{[m]}$, $n\in \I_{0,p^k-1}$, $m\in \I_{0,p^\ell-1}$, 
and relations \ref{eq:rels-graded-dual} for $n\in \I_{0,p^{k}-1}$, $m\in \I_{0,p^\ell-1}$; the comultiplication and braiding are given by \eqref{eq:coprod-graded-dual} and \eqref{eq:braiding-graded-dual};
the set $\{\yt^{[m]} \x^{[n]}\colon n\in \I_{0,p^k-1}, m\in \I_{0,p^\ell-1}\}$ is a basis of $\mathfrak{G}(k,\ell)$. 
\end{coro}

\pf
The projection $\pi\colon \widetilde{\toba} \twoheadrightarrow \mathfrak{G}(k,\ell)$ induces 
$\pi^*\colon \mathfrak{G}(k,\ell) \hookrightarrow \cE$. Let $\beta_{m,n}$ the dual basis of 
$\{x^m y^n\colon n\in \I_{0,p^k-1}, m\in \I_{0,p^\ell-1}\}$. Then $\pi^*(\beta_{m,n}) = \yt^{[m]} \x^{[n]}$ and the result follows.
\epf

\end{document}